\documentclass{amsart}
\usepackage{amsmath}
\usepackage{amssymb}
\usepackage{enumerate}
\usepackage{verbatim}
\newtheorem{theorem}{Theorem}[section]
\newtheorem{lemma}[theorem]{Lemma}
\newtheorem{proposition}[theorem]{Proposition}
\newtheorem{corollary}[theorem]{Corollary}
\newtheorem{thma}{Theorem}

\newcommand{\newnumbered}{\newtheorem}
\newcommand{\newunnumbered}{\newtheorem*}
\newenvironment{proof*}{\begin{proof}}{\end{proof}}
\newcommand{\singlebox}{}
\newcommand{\esinglebox}{}

\newnumbered{assertion}{Assertion}    
\newnumbered{conjecture}{Conjecture}  
\newnumbered{definition}{Definition}[section]
\newnumbered{hypothesis}{Hypothesis}
\newnumbered{remark}[theorem]{Remark}
\newnumbered{note}{Note}
\newnumbered{observation}{Observation}
\newnumbered{problem}{Problem}
\newunnumbered{question}{Question}
\newnumbered{algorithm}{Algorithm}
\newnumbered{example}[theorem]{Example}
\newunnumbered{notation}{Notation} 

\newcommand{\NN}{\mathbb{N}}
\newcommand{\ZZ}{\mathbb{Z}}
\newcommand{\RR}{\mathbb{R}}
\newcommand{\AAA}{\mathcal{A}}
\newcommand{\BBB}{\mathcal{B}}
\newcommand{\CCC}{\mathcal{C}}
\newcommand{\LLL}{\mathcal{L}}
\newcommand{\DDD}{\mathcal{D}}

\newcommand{\FFF}{\mathcal{F}}
\newcommand{\MMM}{\mathcal{M}}
\newcommand{\PPP}{\mathcal{P}}
\DeclareMathOperator{\WWW}{Bow}
\newcommand{\Ms}{\mathcal{M}^\sigma}

\newcommand{\GGG}{\mathcal{G}}
\newcommand{\htop}{h_\mathrm{top}}
\newcommand{\ulim}{\varlimsup}
\newcommand{\llim}{\varliminf}
\newcommand{\eps}{\varepsilon}
\newcommand{\ph}{\varphi}
\newcommand{\wb}{w^\beta}
\newcommand{\phb}{\varphi_r}
\newcommand{\phg}{\varphi_0}

\newcommand{\symdiff}{\bigtriangleup}
\newcommand{\V}{(W)}

\newcommand{\F}{(S)}
\newcommand{\Fp}{(Per)}
\newcommand{\one}{\mathbf{1}}

\DeclareMathOperator{\Per}{Per}
\DeclareMathOperator{\Leb}{Leb}
\DeclareMathOperator{\diam}{diam}

\title[Equilibrium states beyond specification]{Equilibrium states beyond specification and the Bowen property}
\author{Vaughn Climenhaga and Daniel J. Thompson}

\address{Vaughn Climenhaga \\ Department of Mathematics \\ University of Toronto}
\email{vclimenh@math.toronto.edu}
\address{Daniel J. Thompson\\  Department of Mathematics \\The Ohio State University}
\email{thompson@math.psu.edu}

\thanks{V.C.\ is supported by an NSERC Postdoctoral Fellowship.  D.T.\ is supported by NSF grant DMS-$1101576$.}

\begin{document}

\begin{abstract}
It is well-known that for expansive maps and continuous potential functions, the specification property (for the map) and the Bowen property (for the potential) together imply the existence of a unique equilibrium state.  We consider symbolic spaces that may not have specification, and potentials that may not have the Bowen property, and give conditions under which uniqueness of the equilibrium state can still be deduced. Our approach is to ask that the collection of cylinders which are obstructions to the specification property or the Bowen property is small in an appropriate quantitative sense. This allows us to construct an ergodic equilibrium state with a weak Gibbs property, which we then use to prove uniqueness. 
We do not use inducing schemes or the Perron--Frobenius operator, and we strengthen some previous results obtained using these approaches. 
In particular, we consider $\beta$-shifts and show that the class of potential functions with unique equilibrium states strictly contains the set of potentials with the Bowen property. We give applications to piecewise monotonic interval maps, including the family of geometric potentials for examples which have both indifferent fixed points and a non-Markov structure.  
\end{abstract}

\maketitle

\section{Introduction}
An \emph{equilibrium state} for a topological dynamical system $(X,f)$ and a potential $\ph\in C(X)$ is an invariant measure that maximises the quantity $h_\mu(f) + \int \ph\,d\mu$.  For a symbolic space, every continuous function has at least one equilibrium state. We establish uniqueness in a setting that improves 
previous results by simultaneously relaxing the structural requirements on $X$ and the regularity properties of $\ph$. 

It was shown in \cite{Bo7} that uniqueness holds when $(X,f)$ satisfies expansivity and specification and $\ph$ satisfies the Bowen property (Definition~\ref{def:bowen}); in particular, this is true if $X$ is a mixing shift of finite type and $\ph$ is H\"older continuous.  Working in the symbolic setting (where expansivity is automatic), we give weakened versions of the specification property and the Bowen property that still suffice to prove uniqueness, and verify these conditions for specific examples.  In particular, we obtain the following.

\begin{thma}\label{thm:beta}
Let $(X,f)$ be a $\beta$-shift or a shift with the classical specification property (Definition~\ref{def:spec}).  Then there is a class of potentials $\BBB \subset C(X)$, strictly containing the potentials with the Bowen property,
 such that every $\ph \in \BBB$ has a unique equilibrium state.  Furthermore, this equilibrium state has the weak Gibbs property~\eqref{eqn:wkgibbs} and is the weak* limit of the periodic orbit measures
\begin{equation}\label{eqn:permeas}
\frac 1{\sum_{x\in \Per_n} e^{S_n\ph(x)}} \sum_{x\in \Per_n} e^{S_n\ph(x)}\delta_x.
\end{equation}
As an application, for any $\beta$-transformation, every H\"older continuous function has a unique equilibrium state.
\end{thma}

Theorem~\ref{thm:beta} summarises results from \S \ref{sec:betashifts},  \S \ref{sect:nbb}, \S \ref{nb:beta} and \S \ref{betaandF}, and is an application of our more general result, Theorem~\ref{thm:main}, which gives explicit conditions on the language of a shift space which guarantee the existence of a unique equilibrium state. Before we explain these conditions, we further discuss our applications. 
Theorem~\ref{thm:beta} generalises previously known results in two directions. On the one hand, for $\beta$-shifts (which generically do not have specification), uniqueness of the equilibrium state was established by Walters for Lipschitz potentials \cite{Wa2}, and by Denker, Keller and Urbanski~\cite{DKU} for potentials $\ph$ with the Bowen property, but only when $\sup \ph < P(\ph)$. In particular, Theorem~\ref{thm:beta} is the first result to establish uniqueness of equilibrium states for every H\"older continuous function on the $\beta$-shift. 

On the other hand, Theorem~\ref{thm:beta} gives new results for potentials without the Bowen property, even for shifts with specification.  We describe a new class of potentials with unique equilibrium states, generalising a variant of the family of \emph{grid functions} defined by Markley and Paul \cite{MP82,IT10}. This class of potentials includes 
the pioneering examples on the full shift studied by Hofbauer \cite{fH77}. Other results on potentials which are not H\"older continuous have appeared in~\cite{PZ06,hH08}. The following result is a special case of our analysis in \S\S \ref{sect:nbb}--\ref{nb:beta}.
\begin{thma}\label{thm:nonHolder}
Let $X$ be a $\beta$-shift or a shift with the 
specification property, and suppose that $X$ contains the fixed point $0=000\cdots$.  Let $k(x)$ be the number of initial $0$s in the sequence $x\in X$, and let $\ph(x) = \phb(x) + \phg(x)$, where $\phb$ has the Bowen property, $\phg(0)=0$, and for $x \neq 0$, $\phg (x) = a_{k(x)}$ for some sequence $a_n$ satisfying $\lim_{n\to\infty} a_n = 0$.  If $|\sum_{n\geq 1} a_n| = \infty$, then $\ph$ does not have the Bowen property.  Nevertheless, if 
\begin{equation} \label{cond:Hof}
\ph(0) < P(X,\ph)
\end{equation}
then there exists a unique equilibrium state $\mu_\ph$ for $\ph$.  Furthermore, $\mu_\ph$ has the weak Gibbs property~\eqref{eqn:wkgibbs} and is the weak* limit of the periodic orbit measures~\eqref{eqn:permeas}.
\end{thma}
When $X$ has specification, we also obtain results (Theorem~\ref{thm:nonWalters}) where $\phg$ depends in a precise way on the structure of an arbitrary subshift $Y\subset X$. Thus, the class $\BBB$ mentioned in Theorem~\ref{thm:beta} contains many potentials besides those in Theorem~\ref{thm:nonHolder}.

We apply our results to certain piecewise expanding maps of the interval. The following example, which is studied via an application of Theorem~\ref{thm:nonHolder}, is a generalisation of the Manneville--Pomeau map. This example is  developed rigorously in \S \ref{chap:non-symbolic},  and demonstrates the efficacy of Theorem~\ref{thm:nonHolder}. 
\begin{example}\label{eg:betaMP}
Fix $\gamma>0$ and $0<\eps<1$.  Define a piecewise monotonic map of the interval $[0,1]$ by $f(x) = x + \gamma x^{1+\eps} \pmod 1$.  Consider the geometric potential $\ph(x) = -\log |f'(x)|$, which does not have the Bowen property. This system can be modeled by a $\beta$-shift  and the potential $t \ph$ can be modeled by a function satisfying the hypotheses of Theorem~\ref{thm:nonHolder} whenever $t<1$.  As a consequence of Theorem~\ref{thm:nonHolder}, the potential $t \ph$ has a unique equilibrium state for $t<1$ and hence the map $t\mapsto P(t\ph)$ is $C^1$ in this domain. 
\end{example}
When $\gamma=1$, 
we obtain the Manneville--Pomeau map, where the thermodynamics have been thoroughly studied~\cite{PS,U,PW,oS01}, often using the technique of inducing. 
For $\gamma \neq 1$, although a proof using inducing schemes \cite{lsY99,PeSe,BT09} is likely to be possible, a number of technical hypotheses must be verified and to the best of our knowledge, the details  have never been worked out. This example demonstrates that our techniques can be a very useful alternative to inducing.

We now discuss the ideas behind our more general main result, Theorem~\ref{thm:main}, which provides abstract conditions implying uniqueness of the equilibrium state for a potential $\ph$ on a shift space $(X,\sigma)$. We develop the approach that we introduced in \cite{CT}. The idea is that while specification and the Bowen property may not hold on the entire language of the shift space, 
we can still ask for them to hold on a collection of `good' words; then we can require that the set of words that fail to be good is small in an appropriate sense. 

The good words can be chosen to deal with either the failure of the specification property (e.g.\ for $\beta$-shifts), or the failure of the function to satisfy the Bowen property everywhere on the space (e.g.\ the potentials described in Theorem~\ref{thm:nonHolder} on the full shift), or both simultaneously (e.g.\ those same potentials on the $\beta$-shifts).  The ability to deal with both of these non-uniformities concurrently is an important advantage of our approach. 

We use a standard argument to construct an equilibrium measure as the limit of $\delta$-measures supported on a set of orbits corresponding to words of length $n$, with each orbit given a weight proportional to $e^{S_n\ph(x)}$.  A crucial step in our argument is to show that this measure satisfies a Gibbs property on the collection of `good words'. 
This ensures that there is no room for any mutually singular equilibrium measure, and combined with an ergodicity argument, this proves uniqueness. 

We apply our results to expanding piecewise monotonic maps, including the $\beta$-transformation, and also consider cases where the map is non-uniformly expanding, as in Example~\ref{eg:betaMP}.

Our techniques are well adapted to the operation of taking factors. We use results from \cite{CT} to establish uniqueness of equilibrium states for symbolic spaces $(X, \sigma)$ which are factors of $\beta$-shifts. In this case, we prove uniqueness of equilibrium states for potentials $\ph$ which satisfy the Bowen property and the additional hypothesis $\sup \ph - \inf \ph < \htop (X)$.

In \S \ref{results}, we formulate necessary definitions and state Theorem~\ref{thm:main}, which is our main result. In \S \ref{symb-examp}, we apply Theorem~\ref{thm:main} in the symbolic setting, and state various results that prove Theorem~\ref{thm:nonHolder}.  In \S \ref{chap:non-symbolic}, we apply our results to certain interval maps and develop Example~\ref{eg:betaMP} in more detail.  In \S \ref{proof}, we prove Theorem~\ref{thm:main}, and in \S\ref{proof2}, we prove lemmas and propositions from \S\S\ref{symb-examp}--\ref{chap:non-symbolic}.

\section{Definitions and statement of result}\label{results}

\subsection{Notation and general definitions}

A topological dynamical system is a compact metric space $X$ together with a continuous map $f\colon X\to X$.  We restrict our attention to the case where $X \subset A^\NN$ 
 for some finite set $A$, called an \emph{alphabet}, and $f=\sigma$ is the shift map, defined by $\sigma(x)_n = x_{n+1}$. We use the notation $\Sigma^+_b$ for the shift space $\{0, \ldots, b-1\}^\NN$. We use the abbreviation SFT for a shift of finite type.

 Our results (like those in \cite{CT}) apply equally well when $X \subset A^\ZZ$ but all of the examples in this paper are one-sided, so the non-invertible case will be our focus. We recall some basic results and establish our notation. We refer the reader to ~\cite[Chapter 9]{Wa} and \cite{LM} for further background information. 

The \emph{language} of $X$ is the set of finite words that appear in $X$. We denote the language of $X$ by $\LLL$.
To each word $w$ we associate the \emph{cylinder}
\[
[w] = \{ x\in X \mid w = x_1 x_2 \cdots x_{|w|}\},
\]
where $|w|$ denotes the length of the word $w$.  Thus, the language of $X$ can be characterised by $w\in \LLL \Leftrightarrow  [w] \neq \emptyset$. 
Given two words $v=v_1\cdots v_m$ and $w=w_1\cdots w_n$, we write $vw = v_1 \cdots v_m w_1 \cdots w_n$. Given collections of words $\AAA,\BBB \subset \LLL$, we write
\[
\AAA \BBB = \{ vw \in \LLL \mid v\in \AAA, w\in \BBB \}.
\]
Note that only words in $\LLL$ are included in $\AAA\BBB$. Unless otherwise indicated, we use superscripts to index collections of words, and subscripts to index the entries of a given word. That is, if we have a collection of words $\{w\}$, then $w^j$ is the $j^{th}$ word in the collection and $w^j_i$ is the $i^{th}$ entry of the word $w^j$.  We write $\LLL_n$ for the set of words in $\LLL$ with length $n$. Let $C(X)$ denote the space of continuous function on $X$ and given $\ph\in C(X)$, define a function $\ph_n \colon \LLL_n \to \RR$ by
\[
\ph_n(w) = \sup_{x\in [w]} S_n \ph(x),
\]
where $S_n\ph(x) = \ph(x) + \ph(\sigma x) + \cdots + \ph(\sigma^{n-1} x)$.  Given a collection of words $\DDD \subset \LLL$, we write $\DDD_n = \DDD \cap \LLL_n$, and we consider the quantities
\[
\Lambda_n(\DDD,\ph) = \sum_{w\in \DDD_n} e^{\ph_n(w)}.
\]
The (upper capacity) pressure of $\ph$ on $\DDD$ is given by
\[
P(\DDD,\ph) = \ulim_{n\to\infty} \frac 1n \log \Lambda_n(\DDD,\ph).
\]
When $\DDD = \LLL$, we recover the standard definition of topological pressure, and we write $P(\ph)$ or $P(X,\ph)$ in place of $P(\LLL, \ph)$.

Let $\Ms(X)$ denote the space of shift-invariant Borel probability measures on $X$.  We write $h(\mu)$ for the measure-theoretic entropy of $\mu \in \Ms(X)$. The variational principle states that 
\[
P(\ph) = \sup \left\{ h(\mu) + \int \ph\,d\mu \,\Big|\, \mu \in \Ms(X)\right\}.
\]
An invariant probability measure that attains this supremum is called an \emph{equilibrium state} for $\ph$.  
We write $\Per_n$ for the collection of periodic points of period $n$ -- that is, $\Per_n = \{x\in X \mid \sigma^n(x) = x\}$.  Observe that this differs from the notation in~\cite{CT}, where $\Per_n$ denoted points of period \emph{at most} $n$.
\subsection{Specification properties and regularity conditions}
As in~\cite{CT}, we formulate specification properties that apply only to a subset of the language of the space. Our definition applies to naturally defined subsets of the languages of many examples, such as $\beta$-shifts, that do not have specification.

\begin{definition}\label{def:spec}
Given a shift space $X$ and its language $\LLL$, consider a subset $\GGG \subset \LLL$.  Fix $t\in \NN$; any of the following conditions defines a \emph{specification property on $\GGG$} with gap size $t$.
\begin{description}
\item[(W)] For all $m\in \NN$ and $w^1,\dots,w^m\in \GGG$, there exist $v^1,\dots,v^{m-1}\in \LLL$ such that $x := w^1 v^1 w^2 v^2 \cdots v^{m-1} w^m \in \LLL$ and $|v^i| \leq t$ for all $i$.
\item[(S)] Condition \V\ holds, and in addition, the connecting words $v^i$ can all be chosen to have length exactly $t$.
\item[(Per)] Condition \F\ holds, and in addition, the cylinder $[x]$ contains a periodic point of period exactly $|x| + t$.
\end{description}
In the case $\GGG = \LLL$, \F-specification is the well known specification property of the shift. In this case, \F-specification and \Fp-specification are equivalent (this is a folklore result, which follows from the fact that shifts with \F-specification admit synchronizing words \cite{aB88}).
\end{definition}

We now define the regularity condition that we require, which generalises the well known property introduced by Bowen in \cite{Bo7}.
\begin{definition}\label{def:bowen}
Given $n\in \NN$ and $\GGG \subset \LLL$, let
\[
V_n(\GGG,S_n\ph) = \sup \{ |S_n \ph(x) - S_n \ph(y)| \mid x,y\in [w], w\in \GGG_n \}.
\]
A potential $\ph$ has the \emph{Bowen property} on $\GGG$ if $\sup_{n\in \NN} V_n (\GGG,S_n \ph) < \infty$.  Denote the set of such potentials by $\WWW(\GGG)$. If $\ph$ has the Bowen property on $\LLL$, then we just say that $\ph$ has the Bowen property. 
\end{definition}

Note that for shift spaces every H\"older continuous potential has the Bowen property. This is because $V_n(\LLL,S_n\ph) \leq \sum_{k=1}^{n} \sup \{ |\ph(x) - \ph(y)| \mid x,y\in [w], w\in \LLL_k \}$, and H\"older continuity implies that the quantity in the sum decays exponentially in $k$.

\subsection{Main result}
We consider \emph{decompositions} of the language:  collections of words $\CCC^p,\GGG,\CCC^s \subset \LLL$ such that $\CCC^p\GGG\CCC^s = \LLL$. Every word in $\LLL$ can be written as a concatenation of a `good' core (from $\GGG$) with a prefix and a suffix (from $\CCC^p$ and $\CCC^s$).  Given such a decomposition, we consider for every $M\in \NN$ the following `fattened' set of good words
\[
\GGG(M) = \{ uvw \in \LLL \mid u\in \CCC^p, v\in \GGG, w\in\CCC^s, |u|\leq M, |w|\leq M \}.
\] 
Note that $\bigcup_M \GGG(M) = \LLL$, so this gives a filtration of the language.
\begin{thma}\label{thm:main}
Let $(X,\sigma)$ be a subshift on a finite alphabet and $\ph\in C(X)$ a potential.  Suppose there exists collections of words $\CCC^p, \GGG, \CCC^s \subset \LLL$ such that $\CCC^p \GGG \CCC^s = \LLL$ and the following conditions hold:
\begin{enumerate}[(I)]
\item \label{cond:spec}
$\GGG(M)$ has \F-specification for every $M$;
\item \label{cond:bowen}
$\ph \in \WWW(\GGG)$;
\item \label{cond:PR}
The collections $\CCC^s$ and $\CCC^p$ satisfy
\begin{equation}\label{eqn:PR}
\sum_{n\geq 1} \Lambda_n (\CCC^p \cup \CCC^s, \ph) e^{-nP(\ph)} < \infty;
\end{equation}
\end{enumerate}
Then $\ph$ has a unique equilibrium state $\mu_\ph$, which satisfies the following weak Gibbs property:  there exists constants $K', K_M > 0$ such that for every $n\in \NN$ and $w\in \GGG(M)_n$, we have
\begin{equation}\label{eqn:wkgibbs}
K_M \leq \frac{\mu_\ph([w])}{e^{-nP(\ph) + \ph_n(w)}} \leq K'.
\end{equation}
If \F-specification is replaced with \Fp-specification in Condition~\eqref{cond:spec}, then

\begin{equation} \label{permeas}
\mu_\ph = \lim_{n\to\infty} \frac 1{\sum_{x\in \Per_n} e^{S_n\ph(x)}} \sum_{x\in \Per_n} e^{S_n\ph(x)}\delta_x,
\end{equation}
where $\Per_n$ is the collection of periodic orbits of length exactly $n$.
\end{thma}

\begin{remark}
If $P(\CCC^p \cup \CCC^s, \ph) < P(\ph)$, then Condition~\eqref{cond:PR} holds. 
\end{remark}
\begin{remark}
When $\ph=0$, the conditions in the main theorem of \cite{CT} imply the conditions above except with \V-specification in place of \F-specification in \eqref{cond:spec}. The theorem holds true if we assume \V-specification in \eqref{cond:spec} but it leads to some additional technicalities in the proof. The stronger assumption is made purely out of convenience and is 
satisfied by all examples under consideration here. 
\end{remark}
\begin{remark}
Because we assume that $\GGG(M)$ satisfies \Fp-specification for every $M$, we obtain a stronger result concerning the periodic orbit measures than the corresponding result in~\cite{CT}, where we considered measures supported on periodic points of period \emph{at most} $n$.
\end{remark}
\begin{remark} \label{noatoms}
The Gibbs property \eqref{eqn:wkgibbs} shows that $\mu$ is fully supported on $X$. If $X$ is a non-trivial shift space (i.e.\ contains an infinite number of points), this shows, by ergodicity, that $\mu$ has no atoms.
\end{remark}
\begin{remark} 
Theorem \ref{thm:main} applies both when $X$ is one-sided (i.e.\ $X \subset A^\NN$) and two-sided  (i.e.\ $X \subset A^\ZZ$). The role of the prefix collection $\CCC^p$ seems to be much more important in the two-sided case. Indeed, for all of the examples considered in this paper, which are one-sided, $\CCC^p = \emptyset$. In~\cite{CT}, we gave many two-sided examples ($S$-gap shifts and coded systems) where the prefixes are indispensable. We also note that in the two-sided case, the notation $[w]$ refers to the standard two-sided central cylinder. 
\end{remark}
\section{Symbolic Examples} \label{symb-examp}
In this section, we apply Theorem~\ref{thm:main} to symbolic systems.  In \S \ref{sec:betashifts}, we show that every function on a $\beta$-shift which has the Bowen property satisfies the hypotheses of Theorem~\ref{thm:main}. 
In \S \ref{sect:nbb}, we prove Theorem~\ref{thm:nonHolder} (and more) for shifts with specification. In \S \ref{nb:beta}, we prove Theorem~\ref{thm:nonHolder} for $\beta$-shifts. Combining these results yields the full statements of Theorems~\ref{thm:beta} and~\ref{thm:nonHolder}.
\subsection{$\beta$-shifts}\label{sec:betashifts}
Fix $\beta>1$, write $b=\lceil\beta\rceil$, and let $\wb\in \{0,1,\dots,b-1\}^\NN$ be the greedy $\beta$-expansion of $1$ (see \cite{CT,fB89,Pa,Maia} for details). Then $\wb$ satisfies
\begin{equation}\label{wb}
\sum_{j=1}^{\infty} \wb_j \beta^{-j} = 1,
\end{equation}
and has the property that $\sigma^k(\wb) \preceq \wb \text{ for all } k \geq 1$, where $\preceq$ denotes the lexicographic ordering. The $\beta$-shift is defined by
\begin{equation} \label{lexbeta}
\Sigma_\beta = \left \{x \in \{0, 1, \ldots, b-1\} : \sigma^k (x) \preceq \wb  \text{ for all } k \geq 1 \right \}.
\end{equation}
For the rest of this exposition, we assume that $\wb$ is not eventually periodic. This happens for Lebesgue almost every $\beta$, and is the interesting case for our analysis.  Although our methods apply equally well when $\wb$ is eventually periodic, in this case $\Sigma_\beta$ is a sofic shift, and thus the thermodynamic formalism is already well understood.

We showed in~\cite{CT} that the language of $\Sigma_\beta$ can be decomposed as $\LLL = \GGG \CCC^s$,  where $\CCC^s_n = \{\wb_1 \cdots \wb_n\}$. 
We briefly review the construction, and show that $\GGG(M)$ has \Fp-specification for every $M$. 

Every $\beta$-shift can be presented by a countable state directed labeled graph $\Gamma_\beta$, as follows (see \cite{BH,PfS,CT}).  Consider a countable set of vertices labeled $v_1, v_2, \ldots$. For every $i \geq 1$, we draw an edge from $v_i$ to $v_{i+1}$, and label it with the value $\wb_i$. Next, whenever $\wb_i >0$, for each integer from $0$ to $\wb_i-1$, we draw an edge from $v_i$ to $v_1$ labeled by that value. 

The $\beta$-shift can be characterised as the set of sequences given by the labels of infinite paths through the directed graph which start at $v_1$.  For our set $\GGG$, we take the collection of words labeling a path that begins and ends at the vertex $v_1$.  It is clear that
\begin{enumerate}
\item such paths can be freely concatenated, and each one corresponds to a periodic point -- in particular, $\GGG$ has \Fp-specification with $t=0$;
\item $\GGG(M)$ is the set of words labeling finite paths that begin at $v_1$ and terminate at some vertex $v_i$ such that $i \leq M$.
\item if $\tau_M = \max \{$length of shortest path from  $v_i$ to  $v_1 \mid 0 \leq i \leq M\}$, then $\GGG(M)$ has \Fp-specification with $t=\tau_M$. The `gap' can be made to be exactly $\tau_M$ rather than at most $\tau_M$ by padding out with a string of $0$'s based at $v_1$ if necessary. 
\end{enumerate}

It is clear from the graph presentation of $\Sigma_\beta$ that every word in $\LLL$ can be written as a word from $\GGG$ followed by a word in $\CCC^s = \{ \wb_1 \cdots \wb_n \mid n\geq 1 \}$.  From the remarks above, Condition~\eqref{cond:spec} is satisfied.

Suppose $\ph$ has the Bowen property; then Condition~\eqref{cond:bowen} is immediate.  Thus in order to apply Theorem~\ref{thm:main}, it remains only to show that $\ph$ satisfies Condition~\eqref{cond:PR}.  Because $\CCC^s_n$ is a singleton for all $n$, this amounts to checking that
\begin{equation}\label{eqn:betaPR}
\sum_{n\geq 1} e^{S_n \ph(\wb) - nP(\Sigma_\beta, \ph)} < \infty,
\end{equation}
and thus it suffices to show that
\begin{equation}\label{eqn:betaPR2}
\ulim_{n\to\infty} \frac 1n S_n \ph(\wb) < P(\Sigma_\beta, \ph).
\end{equation}
\begin{proposition}\label{prop:Holder}
Suppose $\ph \in \WWW(\Sigma_\beta)$.  Then~\eqref{eqn:betaPR2} holds. 
\end{proposition}
It follows from Proposition~\ref{prop:Holder} and Theorem~\ref{thm:main} that every $\ph \in \WWW(\Sigma_\beta)$ has a unique equilibrium state. In particular, every H\"older continuous potential $\ph$ on $\Sigma_\beta$ has a unique equilibrium state.

\begin{remark}
Equilibrium states for $\beta$-shifts were studied by Walters~\cite{Wa2}, who dealt with the smaller class of Lipschitz potentials, but obtained stronger results than we do regarding properties of the unique equilibrium states.  His approach relies on the Perron--Frobenius operator $\LLL_\ph$ defined by
\[
(\LLL_\ph h)(x) = \sum_{\sigma y=x} e^{\ph(y)} h(y).
\]
In the remarks following~\cite[Corollary 8]{Wa2}, he observes that if a H\"older continuous $\ph$ is such that there exists $M>0$ with
\begin{equation}\label{eqn:cor8}
\LLL_\ph^n \one (x) \leq M e^{nP(\ph)}
\end{equation}
for all $n\geq 0$, then the remainder of the results in~\cite{Wa2} go through. This inequality is a corollary of a key step in the proof of our main theorem: the upper bound in Proposition~\ref{prop:Ln} and the elementary inequality  $\LLL_\ph^n \one(x) \leq \Lambda_n(\LLL,\ph)$ yields~\eqref{eqn:cor8} for every H\"older $\ph$.  Thus in addition to the uniqueness results proved here, we can deduce the following result from~\cite[Theorems 10,13,15]{Wa2}.
\end{remark}

\begin{theorem}\label{thm:after-Walters}
Given a H\"older continuous $\ph$ on a $\beta$-shift $\Sigma_\beta$, there exists a Borel probability measure $\nu$ on $X$ such that $\int \LLL_\ph g \,d\nu = e^{P(\ph)} \int g\,d\nu$ for all $g\in C(\Sigma_\beta)$, and a positive function $h\in C(\Sigma_\beta)$ such that $\LLL_\ph h = e^{P(\ph)} h$ and the following are true:
\begin{enumerate}
\item $\mu_\ph = h\nu$ is a $\sigma$-invariant probability measure and the unique equilibrium state for $\ph$;
\item $e^{-nP(\ph)} \LLL_\ph^n g$ converges uniformly to $h\int g\,d\nu$ for every $g\in C(\Sigma_\beta)$;
\item $\nu \circ \sigma^{-n} \to \mu_\ph$ in the weak* topology;
\item $\LLL_\ph \colon C(\Sigma_\beta) \to C(\Sigma_\beta)$ has spectral radius $e^{P(\ph)}$;
\item the natural extension of $(\Sigma_\beta,\sigma,\mu_\ph)$ is isomorphic to a Bernoulli shift.
\end{enumerate}
Moreover, given two H\"older functions $\ph$ and $\psi$, we have $\mu_\ph = \mu_\psi$ if and only if there exists $g\in C(\Sigma_\beta)$ and $C\in \RR$ such that $\ph - \psi = C + g\circ \sigma - g$.
\end{theorem}


\subsection{Non-Bowen potentials} \label{sect:nbb}
Let $X\subset \Sigma_d^+$ be a shift space, and fix an arbitrary subshift $Y\subset X$.  We describe a class of potentials on $X$ for which the Bowen property fails due to a detoriation in the regularity of the potential at points close to $Y$. We give conditions under which our main theorem can be applied to give a unique equilibrium state. 

Our potentials are similar in spirit to the functions considered in~\cite{fH77,MP82} and \S 7 of \cite{IT10}. A motivating example, described fully in \S\ref{MPmap}, is when $Y$ is the fixed point at $0$, and the function models the geometric potential for the Manneville--Pomeau map.

In general, when $Y$ is a non-trivial subshift, the (lack of) regularity in our class of potentials is allowed to depend in a very precise way on the structure of $Y$, as follows.  Consider the 
set
\[
\FFF(X,Y) :=\{w\in \LLL(X) \setminus \LLL(Y) \mid w_1\cdots w_{|w|-1} \in \LLL(Y)\},
\]
whose elements are forbidden words of minimal length for $Y$ in $X$. 
Observe that $[v] \cap [w] = \emptyset$ for every $v,w\in \FFF(X,Y)$, and define a countable partition of $X$ by
\[
\PPP(X,Y) := \left \{ [w] \mid w \in \FFF(X,Y) \right \} \cup \{Y\}.
\]
For example, if  $X = \{0,1\}^\NN$ and $Y =\{0\}$, then 
\[
\PPP(X,Y)= \{ \{0\}, [1], [01], [001], [00001], \ldots\}.
\]
If $X = \{0,1,2\}^\NN$ and $Y = \{0,1\}^\NN$, then 
\[
\PPP(X,Y)= \{ Y, [2], [02], [12], [002], [012], \ldots\}.
\]
\begin{definition}\label{def:gridfn}
A \emph{grid function} for the partition $\PPP(X,Y)$ is a function $\phg$ that can be written as $\phg = \sum_{w\in \FFF(X,Y)} a_w \one_{[w]}$, where $\{a_w\}$ are real numbers such that $\lim_{n\to\infty} \max_{w\in\FFF(X,Y)_n} a_w = 0$. 
\end{definition}
Let $\AAA = \AAA(X,Y)$ be the set of potential functions $\ph = \phb + \phg$ such that $\phb\in \WWW(\LLL(X))$ (the subscript $r$ denotes `regular') and $\phg = \sum_w a_w \one_{[w]}$ is a grid function for $\PPP(X,Y)$.  In many cases, a function $\ph\in\AAA$ does not have the Bowen property on $X$ (see Proposition~\ref{prop:Bowenornot}).  However, given $\ph \in \AAA$, there is a natural way to choose a collection $\GGG$ on which $\ph$ has the Bowen property. This gives us a natural setting where our main theorem may be applied.

\begin{theorem}\label{thm:nonWalters}
Suppose $X \subset A^\NN$ is a shift with specification on a finite alphabet, 
and let $Y\subset X$ be an arbitrary subshift.  Consider the class of potentials $\BBB = \{ \ph\in \AAA(X,Y) \mid P(Y,\ph) < P(X,\ph)\}$.  Then every $\ph\in \BBB$ has a unique equilibrium state. Furthermore, this equilibrium state has the weak Gibbs property~\eqref{eqn:wkgibbs} and is the weak* limit of the periodic orbit measures
\[
\frac 1{\sum_{x\in \Per_n} e^{S_n\ph(x)}} \sum_{x\in \Per_n} e^{S_n\ph(x)}\delta_x.
\]
\end{theorem}
\begin{proof}
We write down suitable collections $\GGG$ and $\CCC^s$ in order to apply our general theorem.  For any $\ell \geq 0$, we can define the collections
\begin{equation}\label{eqn:GCs}
\begin{gathered}
\GGG^\ell =  \{ w\in \LLL(X) \mid \sigma^{|w|-\ell}(w) \notin \LLL(Y)\}, \\
\CCC^{s, \ell} = \{ w\in \LLL(X) \mid w_k  \cdots w_{k+\ell} \in \LLL(Y) \text{ for all } 1\leq k\leq |w|-\ell \}.
\end{gathered}
\end{equation}
For a suitably chosen $\ell$, we will let $\GGG= \GGG^\ell$ and $\CCC^s = \CCC^{s, \ell}$. Given an arbitrary $w\in \LLL(X)$, we can decompose $w$ as $uv$, where every subword of $v$ with length $\ell$ is in $\LLL(Y)$, and hence $v\in \CCC^s$, while the word comprising the last $\ell$ symbols of $u$ is not in $\LLL(Y)$, and hence $u\in \GGG$. 
\begin{lemma}\label{lem:WaltersonG}
$\ph\in \WWW(\GGG^\ell)$ for every $\ph\in\AAA(X,Y)$ and $\ell \geq 1$.
\end{lemma}
\begin{proof*}
It suffices to show that $\phg\in\WWW(\GGG^{\ell})$.  Given $w\in \GGG^{\ell}$, observe that for all $0 \leq k \leq |w| - \ell$,  $\phg$ is constant on each $[\sigma^k(w)]$.  It follows that for every $x,y\in [w]$, we have
\[
\singlebox
|S_n\phg(x) - S_n\phg(y)| \leq \sum_{j=N-\ell+1}^N |\phg(\sigma^j(x)) - \phg(\sigma^j(y))| \leq 2\ell \|\phg\|.
\esinglebox\qedhere
\]
\end{proof*}
\begin{lemma}\label{lem:lexists}
There exists $\ell$ such that $P(\CCC^{s,\ell},\ph) < P(X,\ph)$.
\end{lemma}
\begin{proof}
For ease of exposition, we break the proof into two cases.

\emph{Case 1: Y is an SFT or there exists an SFT $Z$ such that $Y=Z\cap X$.}
Let $\FFF$ be a finite collection of forbidden words which describes the SFT (see \cite{LM} for details), and take $\ell = \max\{ |w| \mid w \in \FFF\}$. Then $\CCC^{s, \ell} = \LLL(Y)$ and a short calculation combined with the assumption that $\ph \in \BBB$  shows that $P(\CCC^{s,\ell}, \ph) = P(Y, \ph) < P(X, \ph)$. 

\emph{Case 2: $Y$ is an arbitrary subshift.}
Let $Z_\ell$ be the SFT whose forbidden words are all words of length at most $\ell$ in $A^\NN \setminus \LLL(Y)$, and let $Y_\ell = Z_\ell \cap X$.  We have $\bigcap_\ell Y_\ell = Y$.  An easy generalisation of the proof of Proposition 4.4.6 of~\cite{LM}  shows that 
$\lim_{\ell\to\infty} P(Y_\ell, \ph) = P(Y, \ph)$. Since $P(Y, \ph) < P(X, \ph)$, we can choose $\ell$ sufficiently large so that $P(Y_\ell, \ph) < P(X, \ph)$.  Just as in case 1 of the proof, a short calculation shows that $P(\CCC^{s, \ell},\ph) = P(Y_\ell,\ph)$. 
\end{proof}
\begin{remark}\label{rmk:nospecyet}
Up to this point, we have not used the fact that the shift has specification.  In particular, Lemma~\ref{lem:WaltersonG} shows that for \emph{all} subshifts $Y\subset X$ and $\ph\in \AAA(X,Y)$, if $\GGG \subset \GGG^\ell$ for some $\ell\geq 1$, then $\ph\in \WWW(\GGG)$. 
\end{remark}
Returning to the proof of Theorem~\ref{thm:nonWalters}, we let $\ell$ be as in Lemma~\ref{lem:lexists} and let $\GGG,\CCC^s$ be as in~\eqref{eqn:GCs}.  We check the hypotheses of Theorem~\ref{thm:main}: 
Since $\LLL(X)$ has specification, and thus \Fp-specification, $\GGG(M)$ has \Fp-specification for every $M$; Condition~\eqref{cond:bowen} holds by Lemma~\ref{lem:WaltersonG}; and Condition~\eqref{cond:PR} follows from Lemma~\ref{lem:lexists}.  Thus for $\ph\in \BBB$, all the hypotheses of Theorem~\ref{thm:main} are satisfied, so the result follows. 
\end{proof}
The following proposition gives necessary and sufficient conditions for $\ph$ to have the Bowen property. 
\begin{proposition}\label{prop:Bowenornot}
Consider $\ph \in \AAA$ and let $\phg = \sum a_w \one_{[w]}$ be the associated grid function.  Then $\ph\in \WWW(\LLL(X))$ if and only if there is $V < \infty$ such that for every $w\in \LLL(Y)$ and $x\in [w]$, we have $|S_{|w|}\phg(x)| \leq V$.
\end{proposition}
\begin{corollary}\label{cor:Bowenornot}
If $Y=\{0\}$ and $\ph=\phb + \phg\in\AAA$ with $\phg = \sum a_w \one_{[w]}$, then $\ph\in\WWW(\LLL(X))$ if and only if $\sum_n a_{0^n z}$ converges for every $z\in A\setminus \{0\}$.
\end{corollary}
We briefly mention a condition under which it is easy to apply Theorem~\ref{thm:nonWalters}.  Observe that whenever $X$ has specification and $\ph = \phb + \phg \in \AAA(X,Y)$ for some subshift $Y\subset X$, the unique equilibrium state of $\phb$ is fully supported on $X$.  Thus we often have $P(Y,\phb) < P(X,\phb)$; for example, this holds whenever $\phb|_Y$ has a unique equilibrium state on $Y$.  In this case there is a gap between the two pressures, and we put a condition on $\phg$ that guarantees the persistence of this gap for $\ph$.
\begin{corollary}\label{cor:BR}
Let $X$ have specification and let $Y\subset X$ be an arbitrary subshift.  Suppose $\ph = \phb +\phg \in \AAA(X,Y)$ satisfies
\begin{equation}\label{eqn:BR}
- \inf \phg < P(X,\phb) - P(Y,\phb).
\end{equation}
Then $\ph$ has a unique equilibrium state.
\end{corollary}
\begin{proof}
Write $\phg^- = \min(0,\phg)$.  Using the fact that the pressure function is monotonic ~\cite[Theorem 9.7]{Wa} 
and the variational principle, we see that
\begin{multline*}
P(X,\ph) = P(X,\phb + \phg) \geq P(X,\phb + \phg^-) \geq P(X,\phb) - \|\phg^-\| \\
= P(X,\phb) + \inf\phg > P(Y,\phb) = P(Y,\ph),
\end{multline*}
where the last inequality uses~\eqref{eqn:BR}.  Thus Theorem~\ref{thm:nonWalters} applies.
\end{proof}
\begin{remark}
If $\phb=0$ and $Y$ is a single periodic orbit, then \eqref{eqn:BR} follows from the familiar condition that $\sup\ph - \inf\ph < \htop(X)$.
\end{remark}
\subsection{Non-Bowen potentials for $\beta$-shifts} \label{nb:beta} We extend Theorem \ref{thm:nonWalters} to the setting when $X$ is a $\beta$-shift and $Y= \{0\}$. 
This will complete the proof of Theorem~\ref{thm:nonHolder}.
\begin{theorem}\label{lem:pospressenough}
Let $\ph\in \AAA(\Sigma_\beta,\{0\})$ be such that the numbers $a_w$ in Definition~\ref{def:gridfn} depend only on the length of $w$.  If $P(\Sigma_\beta, \ph)> \ph(0)$, then $\ph$ has a unique equilibrium state.  Furthermore, this measure has the weak Gibbs property~\eqref{eqn:wkgibbs} and is the weak*-limit of the periodic orbit measures~\eqref{eqn:permeas}.
\end{theorem}
For the proof of Theorem~\ref{lem:pospressenough}, we define collections $\GGG$ and $\CCC^s$ in order to deal with both the non-Markov structure of $\Sigma_\beta$ and the failure of the Bowen property for $\ph$ simultaneously so that we can apply Theorem \ref{thm:main}.  Recall the presentation of $\Sigma_\beta$ via a graph on a countable vertex set. Let $\GGG$ be the collection of words $w$ which label a path that either:  (1) begins and ends at the base vertex $v_1$ and has the additional property that $w_{|w|} \neq 0$; or (2) begins at $v_1$ and ends at $v_2$.

For the suffix set, take $\CCC^s = \CCC^{s,1}  \cup \CCC^{s,2} \cup \CCC^{s,3}$, where
\begin{gather*}
\CCC^{s,1} = \{ w^\beta_1 w^\beta_2 \cdots w^\beta_m \mid m\geq 1 \} \cdot A \cdot \{0^\ell \mid \ell \geq 0 \},\\
\CCC^{s,2} = \{ w^\beta_2 w^\beta_3 \cdots w^\beta_m \mid m\geq 2 \} \cdot A \cdot \{0^\ell \mid \ell \geq 0 \},\\
\CCC^{s,3} = \{ 0^\ell \mid \ell \geq 0 \}.
\end{gather*}
It is clear that every word in $\LLL$ can be written as a word in $\GGG$ followed by a word in $\CCC^s$. It follows from our earlier analysis of the $\beta$-shift that $\GGG(M)$ has \Fp-specification for every $M$. To prove Theorem \ref{lem:pospressenough} we need only verify that  $\CCC^s$ satisfies Condition~\eqref{cond:PR} of Theorem \ref{thm:main}. This is the content of  \S \ref{posspresproof}.
\subsection{Factors of $\beta$-shifts}
Let $X$ be a subshift factor of a $\beta$-shift. In \cite{CT}, we proved that there is a natural way to write $\LLL(X) = \GGG \CCC^s$, which is inherited from the $\beta$-shift. An easy variation of the argument in \cite{CT} shows that $\GGG (M)$ has \Fp-specification and $P(\CCC^s , 0) = 0$. The following theorem is a corollary of this fact and our main theorem.
\begin{theorem}
Let $X$ be a subshift factor of a $\beta$-shift, and let $\LLL (X) = \GGG \CCC^s$ be the decomposition inherited from the $\beta$-shift. Suppose $\ph \in \WWW(\GGG)$ and $\sup \ph - \inf \ph < \htop (X)$. Then $\ph$ has a unique equilibrium state.
\end{theorem}
\begin{proof*}
We just need to check that $P(\CCC^s, \ph) < P(X, \ph)$.  Since $\#\CCC^s_n$ grows subexponentially, we have $P(\CCC^s, \ph) \leq \sup \ph$. By the variational principle,
\begin{align*}
P(X, \ph) &\geq \htop(X)+ \inf \ph \\  
&\geq \htop(X)-(\sup \ph - \inf \ph) + P(\CCC^s, \ph) \\ &> P(\CCC^s, \ph).\qedhere
\end{align*}
\end{proof*}
\section{Interval maps} \label{chap:non-symbolic}

We prove uniqueness of equilibrium states  for systems that can be well coded by shift spaces meeting the hypotheses of our main theorem. 

\subsection{Piecewise monotonic interval maps}
We consider maps on the interval (perhaps discontinuous) which admit a finite partition such that the map is continuous and monotonic when restricted to the interior of any partition element. That is, let $I=[0,1]$ be the unit interval, and let $f\colon I \to I$ be such that there exists  $p \geq 2$ and $0=a_0 <a_1 < \ldots < a_p=1$ such that writing $I_j = (a_j, a_{j+1})$, the restriction $f|_{I_j}$ is a continuous, monotonic map for every $j$.

Let $S=\{a_0, \ldots, a_p\}$ and $I'= I \setminus \bigcup_{i\geq0}f^{-i}S$. We code $(I,f)$ by the alphabet $A=\{0, \ldots, p-1\}$; the \emph{symbolic dynamics} of $(I, f)$ is defined by the natural coding map $i\colon I' \to A^\NN$, which is given by $i(x)_k = j$ if $f^k(x) \in I_j$. We make the assumption that $i$ is well defined and injective on $I'$.  This is clearly true when $f$ is $C^1$ and satisfies $|f'(x)| \geq \alpha >1$ for all $x$.  When $f$ is assumed only to be increasing on each $I_j$, a sufficient condition for $i$ to be well defined and injective is that $f$ is transitive \cite{FPf,Pa2}.
\begin{definition}
We say that a piecewise monotonic map admits symbolic dynamics if the natural coding map $i\colon I' \to A^\NN$ is injective. The symbolic dynamics of $(I, f)$ is the symbolic space $\Sigma_f\subset A^\NN$ given by
$
\Sigma_f = \overline{i(I')}.
$ 
\end{definition}
We define $\pi\colon \Sigma_f \to I$ by 
$
\pi(x) = \bigcap_{k=0}^\infty f^{-k} \overline{ I_{x_k}},
$
and recall the following important facts.
\begin{enumerate}
\item $(I,f)$ is a topological factor of $(\Sigma_f, \sigma)$, with $\pi$ as a factor map.
\item $\pi$ is injective away from a countable set.
\item Since any measure whose support is contained in a countable set is periodic, $\pi$ can be used to give a measure theoretic isomorphism between $(\Sigma_f, \sigma)$ and $(I,f)$ for any measure which has no atoms.
\item Hence, to prove that $\ph$ has a unique equilibrium state on $(I, f)$, it suffices to show that $\ph \circ \pi$ has a unique equilibrium state which has no atoms. 
\end{enumerate}
There is a natural identification between words $w$ in $\Sigma_f$ and `cylinder sets' $\pi(w) \subset I$, where  $\pi(w)=  \bigcap_{k=0}^{|w|-1} f^{-k} \overline{ I_{w_k}}$. 
Consider the class of functions $\hat C(I) = \{ \ph\colon I\to \RR \mid \ph\circ\pi\in C(\Sigma_f)\}$.  For a function $\ph \in \hat C(I)$, we define a Bowen property which is adapted to the symbolic dynamics. For any $\GGG \subset \LLL(\Sigma_f)$, define
\[
V_n(I, \GGG, S_n\ph) = \sup \{ |S_n \ph(x) - S_n \ph(y)| \mid x,y\in \pi(w), w\in \GGG_n \}.
\]
We define $\WWW_I(\GGG) = \{\ph \in \hat C(I) \mid \sup_n V_n(I, \GGG, S_n\ph) < \infty \}$. For $w \in \LLL(\Sigma_f)$, we define
\[
\ph_n(\pi(w)) = \sup_{x\in \pi(w)} S_n \ph(x).
\]
Given a collection of words $\DDD \subset \LLL(\Sigma_f)$, we write $\DDD_n = \DDD \cap \LLL_n$, and we consider the quantities
\[
\Lambda_n(I, \DDD,\ph) = \sum_{w\in \DDD_n} e^{\ph_n(\pi(w))}.
\]
We define 
\[
P(\ph) = \lim_{n \rightarrow \infty} \frac{1}{n} \log \Lambda_n(I, \LLL(\Sigma_f), \ph).
\]

\begin{theorem} \label{thm:pmm}
Let $f$ be a piecewise monotonic interval map which admits symbolic dynamics, and  $\ph\in \hat C(I)$ be  a potential. Suppose there exist collections of words  $\CCC^p,  \GGG, \CCC^s \subset \LLL(\Sigma_f)$ such that $\CCC^p \GGG \CCC^s = \LLL(\Sigma_f)$ and the following conditions hold:
\begin{enumerate}[(I)]
\item \label{cond:specI}
$\GGG(M)$ has \F-specification for every $M$;
\item \label{cond:bowenI}
$\ph \in \WWW_I(\GGG)$;
\item \label{cond:PRI}
$\sum_{n\geq 1} \Lambda_n (I, \CCC^p \cup \CCC^s, \ph) e^{-nP(\ph )} < \infty$.
\end{enumerate}
Then $\ph$ has a unique equilibrium state $\mu_\ph$.  Furthermore, $\mu_\ph$ is fully supported and satisfies the weak Gibbs property~\eqref{eqn:wkgibbs}.  If each $\GGG(M)$ has \Fp-specification, $\mu_\ph$ is the weak* limit of the periodic orbit measures
\begin{equation}\label{eqn:permeas2}
\frac 1{\sum_{x\in \Per_n} e^{S_n\ph(x)}} \sum_{x\in \Per_n} e^{S_n\ph(x)}\delta_x.
\end{equation}
\end{theorem}
\begin{proof}
The hypotheses of the theorem show that the function $\ph \circ \pi$ on $(\Sigma_f, \sigma)$ satisfies the hypotheses on Theorem \ref{thm:main}. Thus, $\ph \circ \pi$ has a unique equilibrium state on $\Sigma_f$.  We see from Remark~\ref{noatoms} that this equilibrium state has no atoms, and therefore, the discussion above yields a unique equilibrium state for $\ph$ on $(I, f)$.
\end{proof}
We show how to verify these hypotheses for natural classes of functions on some special classes of piecewise monotonic interval maps.
\subsection{$\beta$-transformations and $F$-transformation} \label{betaandF} 
Let $F\colon I \to \RR$ be an increasing $C^1$ map such that $F(0)=0$, and suppose that there exists $\alpha>1$ such that $F'(x) \geq \alpha >1$ for all $x\in I$. Let $f\colon [0,1) \to [0,1)$ be the map given by
\[
f(x) = F(x) \pmod 1.
\]
Following \cite{Wa2,fS95}, we call such a map an $F$-transformation (in~\cite{fH87}, these are called \emph{monotonic mod 1} transformations). When $F(x) = \beta x$ for some $\beta>1$, we recover the definition of the $\beta$-transformation. 

As described in \cite{FPf}, it can be shown that writing $b = \lceil F(1) \rceil$, the space $\Sigma_f \subset \Sigma_b^+$ has a lexicographically maximal element $w$, and $\Sigma_f$ can be characterised as
\[
\Sigma_f = \{ x \in \Sigma_b^+ \mid 0 \preceq \sigma^nx \preceq w \text{ for all } n \geq 0 \}.
\]
In other words, the symbolic dynamics of $(I,f)$ is a $\beta$-shift. Since $f$ is uniformly expanding, it is easy to check that if $\ph \in \hat C(I)$ is H\"older, then $\ph \in \WWW_I( I)$ and thus $\ph \circ \pi \in \WWW(\Sigma_f)$. Proposition \ref{prop:Holder} shows that Condition \eqref{cond:PR} of Theorem \ref{thm:main} holds for $\ph \circ \pi$, and thus $\ph \circ \pi$ has a unique equilibrium state (which has no atoms). Thus, by the comments above, $\ph$ has a unique equilibrium state. In summary, we obtain  
\begin{theorem}
Let $f:[0,1) \mapsto [0,1)$ be an $F$-transformation and $\ph \in \hat C(I)$ have the Bowen property. Then $\ph$ has a unique equilibrium state $\mu_\ph$.  Furthermore, $\mu_\ph$ is fully supported and satisfies the weak Gibbs property~\eqref{eqn:wkgibbs}; it is also the weak* limit of the periodic orbit measures~\eqref{eqn:permeas2}.
\end{theorem}

\subsection{A generalisation of the Manneville-Pomeau maps} \label{MPbeta}
We explain how Theorem \ref{thm:pmm} can be applied to the maps in Example \ref{eg:betaMP}. Recall, we fix $\gamma>0$ and $\eps \in(0,1)$ and define 
\begin{equation} \label{MPmap}
f(x) = x + \gamma x^{1+\eps} \pmod 1.
\end{equation}  
When $\gamma=1$, $f$ is the well-known Manneville--Pomeau map. More generally, $f$ is an $F$-transformation, albeit with only non-uniform expansion. The characterisation of the symbolic dynamics for $F$-transformations holds true in this case, and we see that there exists $\beta >1$ such that the symbolic dynamics of $(I,f)$ is $\Sigma_\beta$. 

Consider the geometric potential $\ph(x) = -\log |f'(x)|$.  
When $t<1$, we show that $t\ph \circ \pi$ is a function on $\Sigma_\beta$ of the type studied in  \S \ref{sect:nbb}. 
\begin{theorem}\label{thm:MP}
Fix $\gamma>0$ and $0<\eps<1$, and let $f$ be the piecewise monotonic interval map defined in~\eqref{MPmap}.  Let $\ph(x) = -\log|f'(x)|$ be the geometric potential.  Then the following are true.
\begin{enumerate}
\item For each $t<1$, the potential $t \ph$ has  a unique equilibrium state $\mu_t$.  The measures $\mu_t$  have the weak Gibbs property~\eqref{eqn:wkgibbs} and are the weak* limit of the periodic orbit measures
\[
\frac{1}{\sum_{x\in\Per_n} ((f^n)'(x))^{-t}} \sum_{x\in\Per_n} ((f^n)'(x))^{-t} \delta_x.
\]
On $(-\infty, 1)$, the pressure function $t\mapsto P(t\ph)$ is $C^1$ and strictly positive.
\item At $t=1$, the function $t\mapsto P(t\ph)$ is not differentiable, and there are at least two distinct ergodic equilibrium states for $\ph$.
\item For $t\geq 1$, we have $P(t\ph) = 0$, and the $\delta$-measure at $0$ is an equilibrium state.
\end{enumerate}
\end{theorem}
\begin{proof}
We show that $\ph \circ \pi \in \AAA(\Sigma_\beta, \{0\})$, and it immediately follows that $t\ph \circ \pi \in \AAA(\Sigma_\beta, \{0\})$ for all $ t$. For $t <1$, we then verify the hypotheses of Theorem~\ref{lem:pospressenough} to establish uniqueness of equilibrium states.  The second and third parts of the theorem follow directly from the inequality $h_\mu(f) \leq \int \log|f'(x)|\,d\mu$ for all ergodic $f$-invariant measures, which holds for this family of maps by~\cite[Theorem 1]{fH91} (see also~\cite[Theorem 7.1]{BJ12}).

Let $x_0$ be the solution of $x + \gamma x^{1+\eps} =1$, and let $x_n$ be the unique number with $0<x_n<x_0$ and $x_{n-1} = f(x_n)$. Note that if $I_0$ is the first interval of monotonicity of $(I, f)$, then $I_0 = (0, x_0)$. Note also that for any $n\geq 1$ and $z\neq 0$, that $\pi(0^nz) \in [x_{n}, x_{n-1}]$. Using the notation conventions of \S \ref{sect:nbb}, we find functions $\phg$ and $\phb$ so that $\ph \circ \pi = \phb + \phg$.
With $X=\Sigma_\beta$ and $Y=\{0\}$, we have $\FFF(X,Y) = \{ 0^n c \mid n\geq 0, c\in \{1,\dots,b-1\} \}$. Define a grid function $\phg$ as in Definition~\ref{def:gridfn} using the values $a_{0^n c} = \ph(x_n)$, so that
\[
\phg(x) = \ph(x_{n}) \text{ whenever } \pi(x) \in (x_{n+1}, x_n],
\]
and $\phg(0) = 0$. Let $\phb = \ph\circ \pi - \phg$. 
The following lemma is proved in \S \ref{sec:lemmasinMP}.
\begin{lemma}\label{lem:phbbowen}
The potential $\phb$ has the Bowen property, while the potential $\ph \circ \pi$ does not.
\end{lemma}
\begin{remark}
The first part of Lemma \ref{lem:phbbowen} is immediate if $\ph$ has the Bowen property with respect to the countable partition $\{(x_0, 1], (x_1, x_0], (x_2, x_1], \ldots\}$, and this is essentially what we prove. For the Manneville--Pomeau map, this was shown by Sarig in \cite[Claim 2]{oS01}. Our proof relies on elegant elementary distortion estimates proved by Young \cite[\S 6]{lsY99}.
\end{remark}
We now assume that $t< 1$ and verify that $P(t\ph)>0$. A key ingredient is the following lemma, whose proof was communicated to us by D.\  Dolgopyat.
\begin{lemma}\label{lem:posexp}
Let $Z_0 = \{x\in I \mid \llim \frac 1n \log (f^n)'(x) = 0 \}$.  Then $\Leb Z_0 = 0$.
\end{lemma}
This is the key tool in deriving the following estimate.
\begin{lemma}\label{lem:gpposent}
For the geometric potential $\ph = -\log f'$, there exists $\lambda>1$ such that $P(t\ph) \geq (1-t)\log \lambda > 0$ for all $t <1$. 
\end{lemma}
Applying Theorem~\ref{lem:pospressenough}, this shows that $t\ph$ has a unique equilibrium state for every $t<1$.  Because the entropy map is upper semi-continuous, this implies that the pressure function is $C^1$. 
\end{proof}
\begin{remark}
The assumption that $0<\eps<1$ is only used in the proofs of Lemmas \ref{lem:posexp} and~\ref{lem:gpposent}. If the inequality $P(t\ph) > 0$  for $t<1$ can be established by other means, then Theorem \ref{thm:MP} follows even if $\eps \geq 1$.
\end{remark}
\section{Proof of Theorem~\ref{thm:main}} \label{proof}
\subsection{Estimates of partition sums}
\begin{lemma}\label{lem:Gnleq}
For every $M \geq 0$, there exists $D_M >0$ such that $\Lambda_n(\GGG(M),\ph) \leq D_M e^{nP(\ph)}$ for all $n$.
\end{lemma}
\begin{proof*}
Using Condition~\eqref{cond:spec}, there is a map $\pi\colon \GGG(M)_n^k \to \LLL_{k(n+t_M)}$ such that $\pi(\vec{w}) = w^1 v^1 w^2 \cdots w^k v^k$ for some $v^i \in \LLL_{t_M}$.  Let $V_M := \sup_n V_n(\GGG(M),S_n\ph)$, and observe that 
\[
V_M \leq 4M\|\ph\| + \sup_k V_k(\GGG,S_k\ph) < \infty,
\]
using the fact that $\ph\in \WWW(\GGG)$.  Furthermore, 
\[
\ph_{k(n+t_M)}(\pi(\vec{w})) \geq
\sum_{j=1}^k \ph_n(w^j) - k(V_M + t_M) \|\ph\|.
\]
Thus we have
\begin{align*}
\Lambda_{k(n+t_M)}(\LLL,\ph) &\geq \sum_{\vec{w}\in \GGG(M)_n^k} e^{(\sum_{j=1}^k \ph_n(w^j)) - k(V_M+t_M\|\ph\|)} \\
&= (\Lambda_n(\GGG(M),\ph) e^{-(V_M+t\|\ph\|)})^k,
\end{align*}
whence
\[
\frac 1{n+t_M} \log \Lambda_n(\GGG(M),\ph) \leq \frac 1{k(n+t_M)} \log \Lambda_{k(n+t_M)}(\LLL,\ph) + \frac {V_M + t_M\|\ph\|}{n+t_M}.
\]
Sending $k\to\infty$ gives
\[
\singlebox
\frac 1n \log \Lambda_n(\GGG(M),\ph) \leq
\frac{n+t_M}n \left( P(\ph) + \frac {V_M + t_M\|\ph\|}{n+t_M} \right).
\esinglebox\qedhere
\]
\end{proof*}
\begin{proposition}\label{prop:Gngeq}
Suppose $\DDD\subset \LLL$ and $C_1>0$ are such that $\Lambda_n(\DDD,\ph) \geq C_1 e^{nP(\ph)}$ for every $n\in \NN$.  Then for every $\delta>0$, there exists $M\in \NN$ such that $\Lambda_n(\DDD\cap \GGG(M),\ph) \geq (1-\delta) \Lambda_n(\DDD,\ph)$ for every $n\in \NN$.
\end{proposition}
\begin{proof*}
Let $a_i = \Lambda_i(\CCC^p \cup \CCC^s, \ph) e^{-iP(\ph)}$, so $\sum_i a_i < \infty$ by Condition~\eqref{cond:PR}.  Every word $y\in \DDD_n$ has the form $uvw$ for some $u\in \CCC_i^p$, $v\in \GGG_j$, and $w\in \CCC_k^s$, where $i+j+k=n$.  If $i\vee k\leq M$, then $y\in \GGG(M)$; thus using the inequality $\ph_n(y) \leq \ph_i(u) + \ph_j(v) + \ph_k(w)$, we have
\[
\Lambda_n(\DDD,\ph) \leq \Lambda_n(\DDD\cap \GGG(M),\ph) + 
\sum_{\substack{i+j+k=n \\ i\vee k > M}} \Lambda_i(\CCC^p,\ph)\Lambda_j(\GGG,\ph) \Lambda_k(\CCC^s,\ph).
\]
Multiplying both sides by $e^{-nP(\ph)}$ and using the result of Lemma~\ref{lem:Gnleq}, we obtain
\[
\Lambda_n(\DDD,\ph) e^{-nP(\ph)} \leq \Lambda_n(\DDD\cap \GGG(M),\ph) e^{-nP(\ph)} + 
\sum_{i\vee k > M} a_i a_k D_0.
\]
Because $\sum a_i < \infty$, we can choose $M$ such that
\[
\sum_{i\vee k>M} a_i a_k D_0 \leq \delta C_1 \leq \delta \Lambda_n(\DDD,\ph) e^{-nP(\ph)},
\]
and so we have
\[
\singlebox
(1-\delta)\Lambda_n(\DDD,\ph) e^{-nP(\ph)} \leq \Lambda_n(\DDD\cap \GGG(M),\ph) e^{-nP(\ph)}. \qedhere
\esinglebox
\] 
\end{proof*}
\begin{proposition}\label{prop:Ln}
There exists $C_2 > 0$ such that $e^{nP(\ph)} \leq \Lambda_n(\LLL,\ph) \leq C_2 e^{nP(\ph)}$ for all $n$.
\end{proposition}
\begin{proof}
Define a map $\pi \colon \bigcup_{i=n}^\infty \LLL_{i} \to \LLL_n$ by $\pi(w) = w_1 \ldots w_n$. Then \[
\ph_{nk}(w) \leq \sum_{i=1}^k \ph_n (\pi(\sigma^{(i-1)n}w)),
\]
which gives
\[
\Lambda_{nk}(\LLL,\ph) \leq \sum_{w\in \LLL_{nk}} e^{\sum_{i=1}^k \ph_n (\pi(\sigma^{(i-1)n}w))}
\leq \Lambda_n(\LLL,\ph)^k.
\]
Thus $\frac 1{nk} \log \Lambda_{nk}(\LLL,\ph) \leq \frac 1n \log \Lambda_n(\LLL,\ph)$, and sending $k\to\infty$ gives the first inequality in the statement of the proposition.

For the second inequality, we apply Proposition~\ref{prop:Gngeq} with $\DDD = \LLL$ to see that there exists $M$ such that $\Lambda_n(\GGG(M),\ph) \geq \frac 12 \Lambda_n(\LLL,\ph)$ for all $n\in \NN$. An application of Lemma~\ref{lem:Gnleq} completes the argument.
\end{proof}

Given $\mu\in \MMM(X)$ and $\DDD_n \subset \LLL_n$, we write $\mu(\DDD_n) = \mu(\bigcup_{w\in \DDD_n} [w])$.
\begin{proposition}\label{prop:Gngeq2}
For every $\gamma>0$, there exist $M\in \NN$ and $C_1>0$ such that if $\nu$ 
is an equilibrium state for $\ph$, then every collection $\DDD_n \subset \LLL_n$ with $\nu(\DDD_n)\geq \gamma$ satisfies
\begin{equation}\label{eqn:Gngeq}
\Lambda_n(\DDD \cap \GGG(M),\ph) \geq C_1 e^{nP(\ph)}.
\end{equation}
\end{proposition}
\begin{proof}
Recall that the entropy of $\nu$ is given by
\[
h(\nu) = \inf_{n\geq 1} \frac 1n \sum_{w\in \LLL_n} -\nu(w) \log \nu(w),
\]
and so for the pressure $P(\ph) = h(\nu) + \int \ph\,d\nu$, we get
\[
nP(\ph) = n\left( h(\nu) + \int \frac 1n S_n\ph\,d\nu \right) 
\leq \sum_{w\in \LLL_n} \nu(w) (\ph_n(w) - \log \nu(w)).
\]
Observe that
\begin{align*}
\sum_{w\in \DDD_n} &\nu(w) (\ph_n(w) - \log \nu(w)) \\
&= \nu(\DDD_n) \sum_{w\in \DDD_n} \frac{\nu(w)}{\nu(\DDD_n)} 
\left( \ph_n(w) - \log \frac{\nu(w)}{\nu(\DDD_n)} - \log \nu(\DDD_n) \right) \\
&\leq \nu(\DDD_n) \log \Lambda_n(\DDD,\ph) - \nu(\DDD_n) \log \nu(\DDD_n),
\end{align*}
where we use the fact that if $a_i \geq 0$, $\sum_i a_i = 1$, and $p_i\in \RR$, then
\[
\sum_i a_i(p_i - \log a_i) \leq \log \sum_i e^{p_i}.
\]
Now writing $\DDD_n^c = \LLL_n \setminus \DDD_n$, the above estimates give
\[
nP(\ph) \leq \nu(\DDD_n) \log \Lambda_n(\DDD,\ph) + \nu(\DDD_n^c) \log \Lambda_n(\DDD^c,\ph) + H,
\]
where $H = \sup_{t\in [0,1]} (-t\log t - (1-t)\log (1-t))$.  Using the fact that $\Lambda_n(\DDD^c,\ph) \leq \Lambda_n(\LLL,\ph)$, we see from Proposition~\ref{prop:Ln} that
\begin{align*}
nP(\ph) &\leq \nu(\DDD_n) \log \Lambda_n(\DDD,\ph) + (1-\nu(\DDD_n))(nP(\ph) + \log C_2) + H \\
&= \nu(\DDD_n)(\log \Lambda_n(\DDD,\ph) - nP(\ph) - \log C_2) + nP(\ph) + \log C_2 + H.
\end{align*}
Now since $\nu(\DDD_n) \geq \gamma$, we have
\[
\log \Lambda_n(\DDD,\ph) - nP(\ph) - \log C_2 \geq -\frac{\log C_2 + H}{\nu(\DDD_n)} \geq -\frac{\log C_2 + H}{\gamma},
\]
which shows that there exists $C_1>0$ such that $\Lambda_n(\DDD,\ph) \geq C_1 e^{nP(\ph)}$. An application of Proposition \ref{prop:Gngeq} completes the proof.
\end{proof}

\subsection{A Gibbs property}
We build an equilibrium state $\mu$ as a limit of $\delta$-measures $\mu_n$ distributed on $n$-cylinders according to the weights given by $e^{\ph_n(w)}$.  To be precise, for each $w\in \LLL$, fix a point $x(w)\in [w]$; then consider the measures defined by
\[
\nu_n = \frac 1{\Lambda_n(\LLL,\ph)} \sum_{x\in \LLL_n} e^{\ph_n(w)} \delta_{x(w)}, \qquad \qquad
\mu_n = \frac 1n \sum_{k=0}^{n-1} (\sigma^*)^k \nu_n,
\]
and let $\mu$ be a weak* limit of the sequence ${\mu_n}$. It is shown in~\cite[Theorem 9.10]{Wa} that $h(\mu) + \int \ph\,d\mu = P(\ph)$, so it remains to show that there can be no other equilibrium states.

\begin{proposition}\label{prop:Gibbs}
For sufficiently large $M$, there exists a constant $K_M$ such that for every $n\in \NN$ and $w\in \GGG(M)$, we have $\mu([w]) \geq K_M e^{-nP(\ph) + \ph_n(w)}$.
\end{proposition}
\begin{proof}
Combining Propositions~\ref{prop:Gngeq} and~\ref{prop:Ln}, we see that for large enough $M$ there exists $C=C(M)>0$ such that $\Lambda_n(\GGG(M),\ph) \geq C e^{nP(\ph)}$ for all $n$.

Fix $w\in \GGG(M)_n$.  We estimate $\mu_m(w)$ for large $m$ by first estimating $\nu_m(\sigma^{-k}(w))$.  Let $t_M$ be the gap size in the specification property, and let $\ell_1 = k-t_M$ and $\ell_2 = m-k-t_M-n$.  From Condition~\eqref{cond:spec}, for every $v^i\in \GGG(M)_{\ell_i}$ there exist words $u^1,u^2\in \LLL_{t_M}$ such that $v^1 u^1 w u^2 v^2 \in \LLL_m$.  Since different choices of $v^1, v^2$ give different words in $\LLL_m$, this defines an injective map $\pi\colon \GGG(M)_{\ell_1}\times \GGG(M)_{\ell_2} \to \LLL_m$. Note that
\[
\ph_m(\pi(v^1, v^2)) \geq \ph_{\ell_1}(v_1) + \ph_n(w) + \ph_{\ell_2}(v_2) - 2t_M \| \ph \|-3V,
\] and this allows us to estimate that
\begin{align*}
\nu_m(\sigma^{-k}(w)) &\geq \frac{\sum_{v^1\in \GGG(M)_{\ell_1}, v^2\in \GGG(M)_{\ell_2}} e^{\ph_m(\pi(v^1,v^2))}}{\Lambda_m(\LLL,\ph)} \\
&\geq \frac{\Lambda_{\ell_1}(\GGG(M),\ph) e^{\ph_n(w)} \Lambda_{\ell_2}(\GGG(M),\ph) e^{-(3V + 2t_M\|\ph\|)}}
{\Lambda_m(\LLL,\ph)}.
\end{align*}
Using the lower bound given above for $\Lambda_{\ell_i}(\GGG(M),\ph)$ together with the result of Proposition~\ref{prop:Ln}, we have
\begin{align*}
\nu_m(\sigma^{-k}([w])) &\geq \frac{C^2 e^{(\ell_1 + \ell_2)P(\ph)} e^{\ph_n(w)} e^{-(3V+2t_M\|\ph\|)}}
{C_2 e^{mP(\ph)}} \\
&= C^2 C_2^{-1} e^{-(3V+2t_M\|\ph\|)} e^{-2t_MP(\ph)} e^{-nP(\ph)+\ph_n(w)}.
\end{align*}
Since this holds for all $k$ and $m$, we are done.
\end{proof}
The Gibbs property above shows that if $\mu$ is ergodic (which we will verify shortly), then $\mu$ is not atomic.
\begin{lemma}\label{lem:upGibbs}
There exists $C_3>0$ such that for every $n\in \NN$ and $w\in \LLL_n$, we have $\mu(w) \leq C_3 e^{-nP(\ph) + \ph_n(w)}$.
\end{lemma}
\begin{proof}
Fix $m > n$ and $k<m-n$.  Using Proposition~\ref{prop:Ln}, we have
\begin{align*}
\nu_m(\sigma^{-k}([w])) &\leq \frac{\Lambda_k(\LLL,\ph) e^{\ph_n(w)} \Lambda_{m-k-n}(\LLL,\ph)}
{\Lambda_m(\LLL,\ph)} \\
&\leq \frac{C_2^2 e^{(m-n)P(\ph)}e^{\ph_n(w)}}{e^{mP(\ph)}},
\end{align*}
and the result follows upon passing to the limit.
\end{proof}
Using the Gibbs property, we show that $\mu$ is concentrated on cylinders from $\GGG(M)$ in the following sense.
\begin{lemma}\label{lem:mostinGM}
Let $\delta_1>0$.  There exists $M$ such that for all $n$, any subset $\DDD_n \subset \LLL_n$ satisfies $\mu(\DDD_n \cap \GGG(M)) \geq \mu(\DDD_n) - \delta_1$.
\end{lemma}
\begin{proof}
By Proposition~\ref{prop:Gngeq}, there exists $M$ such that for all $n$,
\[
\Lambda_n(\GGG(M)^c,\ph) \leq (\delta_1 C_3^{-1} C_2^{-1}) \Lambda_n(\LLL,\ph).
\]
Together with Proposition~\ref{prop:Ln} and Lemma~\ref{lem:upGibbs}, this gives
\[
\mu(\GGG(M)^c) \leq C_3 \Lambda_n(\GGG(M)^c,\ph) e^{-nP(\ph)} \leq (\delta_1 C_3^{-1} C_2^{-1}) C_3 C_2 = \delta_1,
\]
and the result follows.
\end{proof}
\subsection{Ergodicity of $\mu$}
\begin{proposition}\label{prop:partmix}
If two measurable sets $P,Q\subset X$ both have positive $\mu$-measure, then $\ulim_{n\to\infty} \mu(P \cap \sigma^{-n}Q) > 0$.
\end{proposition}
\begin{proof}
We start by considering the case where $P$ and $Q$ are cylinders corresponding to words in $\GGG(M)$.
\begin{lemma}\label{lem:GMcylmix}
For all sufficiently large $M$, there exists $E_M > 0$ such that if $u,v\in \GGG(M)$, then $\llim_{n\to\infty} \mu([u] \cap \sigma^{-n}[v]) \geq E_M \mu(u) \mu(v)$.
\end{lemma}
\begin{proof}
As in the proof of Proposition~\ref{prop:Gibbs}, we take $M$ and $C=C(M)$ such that $\Lambda_n(\GGG(M),\ph) \geq Ce^{nP(\ph)}$ for all $n$.  Now fix $u,v\in \GGG(M)$, let $m\in \NN$ be large and fix $k\leq m$.  We estimate
\[
(\nu_m \circ \sigma^{-k})([u] \cap \sigma^{-n}[v]) = \nu_m(\sigma^{-k}[u] \cap \sigma^{-(n+k)}[v]).
\]
Write $\ell_1 = k-t_M$, $\ell_2 = n-|u|-2t_M$, and $\ell_3 = m-n-k-|v|-t_M$ and notice that $\ell_1+\ell_2+\ell_3-m = -(4t_M + |u|+|v|)$. Using Condition \eqref{cond:spec}, for every $(w^1, w^2, w^3) \in \GGG_{\ell_1} \times \GGG_{\ell_2} \times \GGG_{\ell_3}$ 
there exist $x^1,x^2,x^3,x^4\in \LLL_{t_M}$ such that
\[
w^1 x^1 u x^2 w^2 x^3 v x^4 w^3 \in \LLL_m.
\]
As in Proposition~\ref{prop:Gibbs}, it follows that
\begin{align*}
(\nu_m &\circ \sigma^{-k})([u] \cap \sigma^{-n}[v]) \\
&\geq \frac{\Lambda_{\ell_1}(\GGG(M),\ph) e^{\ph_{|u|}(u)} \Lambda_{\ell_2}(\GGG(M),\ph)
e^{\ph_{|v|}(v)} \Lambda_{\ell_3}(\GGG(M),\ph) e^{-(3V + 4t_M \|\ph\|)}}
{\Lambda_m(\LLL,\ph)} \\
&\geq \frac{C^3 e^{(\ell_1+\ell_2+\ell_3)P(\ph)} e^{\ph_{|u|}(u)} e^{\ph_{|v|}(v)}e^{-(3V + 4t_M \|\ph\|)}}
{C_2 e^{mP(\ph)}} \\
&\geq C^3 C_2^{-1} e^{-|u|P(\ph) + \ph_{|u|}(u)} e^{-|v|P(\ph) + \ph_{|v|}(v)} e^{-4t_M P(\ph)} e^{-(3V + 4t_M \|\ph\|)} \\
&\geq C^3 C_2^{-1} e^{-4t_M P(\ph)} e^{-(3V + 4t_M \|\ph\|)} \mu(u) \mu(v),
\end{align*}
where the last inequality follows from Lemma~\ref{lem:upGibbs}.  This holds for all $k$ and $m$, whence we have the desired result.
\end{proof}
Given $P\subset \GGG(M)_m$, let $[P] = \bigcup_{w\in P}[w]$.  Then for $P,Q\subset \GGG(M)_m$, we have
\[
\mu([P] \cap \sigma^{-n}[Q]) = \sum_{w\in P, w'\in Q} \mu([w] \cap \sigma^{-n}[w']),
\]
and using Lemma~\ref{lem:GMcylmix}, we see that
\begin{equation}\label{eqn:partmix}
\llim_{n\to\infty} \mu([P] \cap \sigma^{-n}[Q]) \geq E_M \mu(P)\mu(Q).
\end{equation}
To complete the proof of Proposition~\ref{prop:partmix}, we let $P,Q\subset X$ be any measurable sets with positive $\mu$-measure, and take $0 < \delta_1 < \mu(P) \wedge \mu(Q)$.  Let $M$ be as in Lemma~\ref{lem:mostinGM}.

Fix $\eps>0$ and choose sets $U,V$ that are unions of cylinders of the same length, say $m$, and for which $\mu(U\symdiff P) < \eps$ and $\mu(V \symdiff Q) < \eps$.  Let $U' = U \cap [\GGG(M)_m]$ and $V' = V\cap [\GGG(M)_m]$; then by Lemma~\ref{lem:mostinGM}, we have $\mu(U') > \mu(U) - \delta_1$ and $\mu(V') > \mu(V) - \delta_1$.  Furthermore, using~\eqref{eqn:partmix}, we see that
\[
\llim_{n\to\infty} \mu(U' \cap \sigma^{-n}(V')) \geq E_M \mu(U') \mu(V'),
\]
whence
\begin{equation}\label{eqn:partmix2}
\llim_{n\to\infty} \mu(U \cap \sigma^{-n}(V)) \geq E_M (\mu(U)-\delta_1)(\mu(V)-\delta_1).
\end{equation}
Finally, we observe that
\begin{align*}
|\mu(U\cap \sigma^{-n}(V)) - \mu(P\cap \sigma^{-n}(Q))| &\leq
\mu((U\cap \sigma^{-n}(V)) \symdiff (P\cap \sigma^{-n}(Q))) \\
&\leq \mu((U\symdiff P) \cap \sigma^{-n}(V \symdiff Q)) < \eps
\end{align*}
for every $n$, which together with~\eqref{eqn:partmix2} implies
\[
\llim_{n\to\infty} \mu(P \cap \sigma^{-n}(Q)) \geq E_M(\mu(P) - \delta_1)(\mu(Q)-\delta_1) - \eps.
\]
Since $\eps>0$ was arbitrary, this completes the proof.
\end{proof}
\subsection{Uniqueness of $\mu$}

Let $\mu$ be the ergodic equilibrium state constructed in the previous sections, and suppose that some ergodic measure $\nu\perp \mu$ is also an equilibrium state.  Let $\DDD$ be a collection of words such that $\mu(\DDD_n)\to 0$ and $\nu(\DDD_n) \to 1$. 
Applying Proposition~\ref{prop:Gngeq2}, we may assume that $M$ and $C_1>0$ are such that
\[
\Lambda_n(\DDD \cap \GGG(M),\ph) \geq C_1 e^{nP(\ph)}
\]
for every $n$.  Using the Gibbs property from Proposition~\ref{prop:Gibbs}, we have
\begin{align*}
\mu(\DDD_n) &\geq \mu(\DDD_n \cap \GGG(M)) 
\geq \sum_{w\in \DDD_n \cap \GGG(M)} K_M e^{-nP(\ph) + \ph_n(w)} \\
&= K_M e^{-nP(\ph)} \Lambda_n(\DDD \cap \GGG(M),\ph) \geq K_M C_1 > 0,
\end{align*}
which contradicts the fact that $\mu(\DDD_n) \to 0$.  This contradiction implies that every equilibrium state $\nu$ for $\ph$ is absolutely continuous with respect to $\mu$, and since $\mu$ is ergodic, this in turn implies that $\nu=\mu$, which completes the proof of Theorem~\ref{thm:main}.

\subsection{Characterisation of $\mu$}
We claim that if  $\GGG(M)$ has \Fp-specification for all $M$, then $\mu$ can be characterised by the equation \eqref{permeas}. The argument is similar to the one described in more detail in \S5.6 of \cite{CT}. The key point is that $\sum_{x\in \Per_{n+\tau}} e^{S_n\ph(x)}$ can be controlled from below by $ \Lambda_n(\GGG(M), \ph)$, where $\tau$ is the gap size in the specification property of $\GGG(M)$. This can be used to show that any limit measure of the sequence of measures in \eqref{permeas} is an equilibrium measure for $\ph$. By uniqueness of the equilibrium measure, this establishes \eqref{permeas}.

\section{Proofs of results from Sections \ref{symb-examp} and \ref{chap:non-symbolic}}\label{proof2}
\subsection{Proof of Proposition \ref{prop:Holder}}
Let $\Sigma \subset \Sigma_\beta$ be an SFT containing $0$.  
This is possible by fixing $N\in \NN$ and taking the shift space defined by all the paths on the graph presentation of $\Sigma_\beta$ which never leave the first $N$ vertices of the graph.  Then there is a unique equilibrium state $\mu'$ for $\ph|_{\Sigma}$, and $\mu'$ is fully supported on $\Sigma$.  In particular, $\mu'$ is non-atomic, and so
\[
P(\Sigma_\beta, \ph) \geq P(\Sigma, \ph) = h_{\mu'}(f) + \int \ph\,d\mu' > \ph(0).
\]
Let $V = \max \{ \sup_n V_n(\LLL,S_n\ph), \|\ph\|_\infty \}$, and fix $\delta>0$ such that
\begin{equation}\label{eqn:gap}
\ph(0) + 8\delta V < P(\Sigma_\beta, \ph).
\end{equation}
Fix $n\geq 1$ and let $w=\wb_1 \cdots \wb_n$ be the unique word in $\CCC^s_n$.  Let $a_n$ be the number of non-zero entries in $w$; that is, $a_n = \# \{ i \in \{1,\dots n\} \mid w_i \neq 0 \}$.  For every $0\leq k \leq a_n$, let $\AAA_k$ be the set of words obtained by changing precisely $k$ of those entries to $0$.  Observe that $\AAA_k \subset \LLL_n$ (this follows from the characterisation of $\Sigma_\beta$ in terms of the lexicographic ordering).  Furthermore, for each $v\in \AAA_k$, we have
\begin{align*}
w &= w^1 x_1 w^2 x_2 \cdots x_k w^{k+1}, \\
v &= w^1 0 w^2 0 \cdots 0 w^{k+1},
\end{align*}
where $w^i \in \LLL$ and $x_i \in A$.  (Recall that $A$ is the alphabet of the shift.)  Writing $\Phi(u) = \ph_{|u|}(u)=\sup_{x\in[u]} S_{|u|}\ph(x)$ for $u\in \LLL$, we have
\begin{align*}
|\Phi(w) - (\Phi(w^1) + \Phi(x_1) + \Phi(w^2) + \Phi(x_2) + \cdots + \Phi(w^k))| &\leq (2k+1)V, \\
|\Phi(v) - (\Phi(w^1) + \Phi(0) + \Phi(w^2) + \Phi(0) + \cdots + \Phi(w^k))| &\leq (2k+1)V,
\end{align*}
and so
\[
|\Phi(v) - \Phi(w)| \leq (4k+2)V + \sum_{i=1}^k |\Phi(x_i) - \Phi(0)| \leq (5k+2)V.
\]
We use the bound $|\Phi(v) - \Phi(w)| \leq 7kV$.  There are ${a_n \choose k}$ distinct words in $\AAA_k$, and so
\[
\sum_{v\in \AAA_k} e^{\Phi(v)} \geq {a_n \choose k} e^{\Phi(w) - 7kV}.
\]
Summing over all $k$ gives
\[
\Lambda_n(\LLL,\ph) \geq \sum_{k=0}^{a_n} {a_n \choose k} e^{\Phi(w)} e^{-7kV} = e^{\Phi(w)} (1+ e^{-7V})^{a_n}.
\]
In particular, if $a_n \geq \delta n$, then we have
\begin{equation}\label{eqn:gap2}
\frac 1n \log \Lambda_n(\LLL,\ph) \geq \frac 1n \ph_n(w) + \delta\log(1+e^{-7V}).
\end{equation}
On the other hand, if $a_n < \delta n$, then we can use a similar argument to compare $\ph_n(w)$ and $n\ph(0)$, obtaining
\[
\ph_n(w) \leq n\ph(0) + 7\delta n V.
\]
Using~\eqref{eqn:gap}, this gives
\[
\frac 1n \ph_n(w) \leq \ph(0) + 7\delta V < P(\Sigma_\beta, \ph) - \delta V,
\]
which together with~\eqref{eqn:gap2} shows that
\[
\ulim_{n\to\infty} \frac 1n S_n \ph(w^\beta) \leq P(\Sigma_\beta, \ph) - \delta \min \{ V, \log (1+e^{-7V})\} < P(\Sigma_\beta, \ph),
\]
establishing~\eqref{eqn:betaPR2} and hence~\eqref{eqn:betaPR}.
\subsection{Proof of Proposition~\ref{prop:Bowenornot} and Corollary~\ref{cor:Bowenornot}}
\begin{proof}[of  Proposition \ref{prop:Bowenornot}]
It suffices to restrict our attention to the grid functions $\phg$.  Fix $w\in \LLL(X)_N$ and let
\[
m = \min \{ k \mid \sigma^k(w) \in \LLL(Y) \}.
\]
Then for all $0 \leq k < m$, we have $\sigma^k(w) \in \LLL(X) \setminus \LLL(Y)$, whence $\phg$ is constant on $[\sigma^k(w)]$.

Now suppose that $V$ exists such that the bound in the hypothesis holds.  We see that for every $x,y\in [w]$ we have 
\[
|S_N\phg(x)- S_N\phg(y)| = |S_{N-m} \phg(\sigma^m(x)) - S_{N-m} \phg(\sigma^m(y))| < 2V.
\] 
For the converse, fix $V\in \RR$ and take $w\in \LLL(Y)_N$ and $x\in [w]$ such that $|S_N \phg(x)| > V$.  Take $y\in Y\cap [w]$, and note that $\phg(\sigma^k(y)) = 0$ for all $k\geq 0$.  Consequently, we have
\[
V_N(\LLL(X),S_N\phg) \geq |S_N \phg(x) - S_N\phg(y)| > V.
\]
Since $V$ was arbitrary, this shows that $\phg\notin \WWW(\LLL(X))$.
\end{proof}
\begin{proof}[of  Corollary \ref{cor:Bowenornot}]
In this case, $\FFF(X,Y) = \{ 0^n z \mid n\geq 0, z\in A\setminus \{0\}\}$.  If $x\in [0^nz]$, then $\sigma(x) \in [0^{n-1}z]$, so for $x\in [0^Nz]$ we have $S_N\phg(x) = \sum_{n=1}^N a_{0^nz}$. If $V=\max \{ |\sum_{n\geq 0} a_{0^n z}| \mid z\in A\setminus \{0\} \}< \infty$, it follows from Proposition~\ref{prop:Bowenornot} that $\phg$ has the Bowen property.  Conversely, if $V=\infty$, it is clear that $S_N\phg(x)$ can be arbitrarily large.
\end{proof}
\subsection{Proof of Theorem \ref{lem:pospressenough}} \label{posspresproof} 
Let $\GGG$ and $\CCC^s$ be as in the discussion following the statement of Theorem~\ref{lem:pospressenough}.  We show that $P(\CCC^s,\ph) < P(\ph)$ in order to apply Theorem~\ref{thm:main}.  Because $\#\CCC^s_n$ grows subexponentially, it will suffice to show that there exists $\delta'>0$ such that
\begin{equation}\label{eqn:sumless}
\frac 1n S_n\ph(x) \leq P(\ph) - \delta'
\end{equation}
for every sufficiently large $n$ and every $x\in [w]$ for $w\in \CCC^s_n$.  
The proof of~\eqref{eqn:sumless} is similar to the proof of Proposition~\ref{prop:Holder}, but greater care must be taken since $\ph$ does not have the Bowen property.  We write $\ph = \phb + \phg$, where the following properties hold:
\begin{enumerate}
\item there exists $V>0$ such that $V_n(\LLL,S_n\phb) \leq V$ for all $n$;
\item $\phg = \sum_{\ell\geq 0} a_\ell \one_{[0^\ell A^+]}$ for some sequence $a_\ell$ tending monotonically to $0$, where $A^+ = A\setminus \{0\}$.
\end{enumerate}
Write $s_\ell = \sum_{j=0}^\ell |a_j|$; then $s_\ell = |S_{\ell+1} \phg(x)|$ for all $x\in [0^\ell A^+]$.  Observe that because $a_j$ converges monotonically to $0$, so does the sequence $\frac 1\ell s_\ell$.  In particular, using the assumption that $\ph(0) < P(\ph)$, there exists $L$ such that
\begin{equation}\label{eqn:Lbig}
\ph(0) + \frac 1L s_L < P(\ph).
\end{equation}
Fix $\delta>0$ such that
\begin{equation}\label{eqn:deltasmall}
\ph(0) + \frac 1L s_L + 8\delta V + \delta s_L < P(\ph),
\end{equation}
and fix $\gamma>0$ such that
\begin{equation}\label{eqn:gammasmall}
\delta' := -\frac \delta4\gamma \log \gamma - \gamma(7V + s_L) > 0.
\end{equation}
Fix $n$ large enough such that $\frac {3C}n \log n \leq \delta'$, where $C>0$ is a constant which, from Stirling's approximation formula, satisfies
\begin{equation}\label{eqn:stirling0}
|\log m! - (m\log m - m)| \leq C\log m
\end{equation}
for all $m$. 
We fix $w\in \LLL_n$ and show that~\eqref{eqn:sumless} holds for every $x\in [w]$.  Write 
\[
w = 0^{\ell_0} x_1 0^{\ell_1} x_2 \cdots x_k 0^{\ell_k}
\]
for some $x_i \in A^+$ and $\ell_i \geq 0$; then 
for every $x\in [w]$, 
\begin{equation}\label{eqn:Snxinw}
S_n \phg(x) = \pm(s_{\ell_0} + s_{\ell_1} + \cdots + s_{\ell_{k-1}}) + S_{\ell_k} \phg(0^{\ell_k} \cdot \sigma^n(x)),
\end{equation}
where the choice of sign depends only on whether the sequence $a_n$ takes positive or negative values.  
Consider the set of indices
\[
Q = \{ 1\leq i < k \mid \ell_i < L \}.
\]
\begin{lemma}\label{lem:lotsofshorts}
If $\#Q \leq \delta n$, then~\eqref{eqn:sumless} holds for every $x\in [w]$.
\end{lemma}
\begin{proof}
By the argument in the proof of Proposition~\ref{prop:Holder}, we have
\begin{equation}\label{eqn:Snphir}
S_n\phb(x) \leq n\ph(0) + 7\delta nV,
\end{equation}
so we must estimate $S_n\phg(x)$.  We see from~\eqref{eqn:Snxinw} that
\[
S_n\phg(x) \leq s_{\ell_0} + s_{\ell_1} + \cdots + s_{\ell_k} 
\leq \delta n s_L + \sum_{i=0}^k \frac {\ell_i}Ls_L 
\leq \delta n s_L + n\frac 1L s_L,
\]
where the second inequality uses the fact that if $\ell_i \geq L$, then $\frac 1{\ell_i} s_{\ell_i} \leq \frac 1 L s_L$.  Together with~\eqref{eqn:Snphir} and~\eqref{eqn:deltasmall}, this completes the proof of the lemma.
\end{proof}
Thanks to Lemma~\ref{lem:lotsofshorts}, it only remains to consider the case where $w\in\LLL_n$ is such that $\#Q \geq \delta n$.  Writing $Q = \{i_1, i_2, \dots, i_{\#Q}\}$, where $i_1 < i_2 < \cdots$, consider the subset $Q' = \{i_1, i_3, i_5, \dots\} \subset Q$.  Then we have $\# Q' \geq \frac 12 \delta n$ and $|j-i| \geq 2$ for all $i,j\in Q'$.

Given $P\subset Q'$, let $w'$ be the word obtained from $w$ by changing $x_i$ to $0$ whenever $i\in P$.  It follows from the characterisation  \eqref{lexbeta} of $\Sigma_\beta$ that $w'\in \LLL$.  

Given $x\in [w]$ and $y\in [w']$ we have as in Proposition~\ref{prop:Holder} that
\begin{equation}\label{eqn:Snphir2}
|S_n\phb(x) - S_n\phb(y)| \leq 7(\#P) V.
\end{equation}
Changing $x_i$ to $0$ results in  a block of $0$s of length $\ell_{i-1} + \ell_i + 1$ (using the fact that $P$ does not contain two consecutive integers).  Using~\eqref{eqn:Snxinw}, this gives us the following estimate for $\phg$:
\begin{equation}\label{eqn:sell}
|S_n\phg(x) - S_n\phg(y)| \leq \sum_{i\in P} |s_{\ell_{i-1} + \ell_i + 1} - s_{\ell_{i-1}} - s_{\ell_i}|.
\end{equation}
Furthermore, we have
\[
|s_{\ell_{i-1} + \ell_i + 1} - s_{\ell_{i-1}} - s_{\ell_i}| = 
\sum_{j=1}^{\ell_i+1} \left (|a_j| - |a_{\ell_{i-1} + j}| \right) \leq \sum_{j=1}^{\ell_i+1} |a_j| = s_{\ell_i+1},
\]
whence~\eqref{eqn:sell} yields
\[
|S_n\phg(x) - S_n\phg(y)| \leq \sum_{i\in P} s_{\ell_i+1} \leq (\#P) s_L.
\]
Together with~\eqref{eqn:Snphir2} and the observation that there are ${\#Q' \choose m}$ ways to choose $P$ with $\#P = m$ for $0 \leq m \leq \#Q'$, we see that
\begin{equation}\label{eqn:logL1}
\Lambda_n(\LLL,\ph) \geq {\#Q' \choose m} e^{S_n\ph(x) - m(7V + s_L)}.
\end{equation}
An elementary calculation shows that for all $m$ and $N$, \eqref{eqn:stirling0} yields
\begin{equation}\label{eqn:stirling}
\log {N\choose m} \geq N \left(-\frac m N \log \frac mN \right) - 3C\log N.
\end{equation}
Taking logarithms in~\eqref{eqn:logL1} and applying~\eqref{eqn:stirling} yields
\[
\log\Lambda_n(\LLL,\ph) \geq \#Q'\left(-\frac m{\#Q'} \log \frac m{\#Q'}\right) - 3C\log \#Q' + S_n\ph(x) - m(7V+s_L)
\]
for every $0 \leq m \leq \#Q'$.  Using the inequality $\frac 12\delta n \leq \#Q' \leq n$ and choosing $m$ such that $\gamma \#Q' \leq m \leq 2\gamma \#Q'$, we obtain
\begin{equation}\label{eqn:logL3}
\log\Lambda_n(\LLL,\ph) \geq \frac 12\delta n (-\gamma \log \gamma) - 3C \log n + S_n\ph(x) - 2\gamma n(7V+s_L).
\end{equation}
In particular, applying~\eqref{eqn:gammasmall} gives
\begin{align*}
\frac 1n\log \Lambda_n(\LLL,\ph) &\geq \frac 1n S_n\ph(x) + \frac 12 \delta (-\gamma\log \gamma) - 2\gamma (7V+s_L) - \frac {3C}n\log n \\
&= \frac 1n S_n\ph(x) + 2\delta' - \frac{3C}n\log n.
\end{align*}
Thus all sufficiently large values of $n$ satisfy
\[
\frac 1n\log \Lambda_n(\LLL,\ph) \geq \frac 1n S_n\ph(x) + \delta',
\]
and~\eqref{eqn:sumless} follows.
\subsection{Proofs of lemmas used in Theorem~\ref{thm:MP}}
\label{sec:lemmasinMP}

For any $\gamma > 0$, the behaviour of the map $f$ near the fixed point $0$ is described by the discussion in~\cite[\S 6]{lsY99}. The following lemma, which is proved by an elegant elementary argument is crucial for our analysis.
\begin{lemma}[({\cite[Lemma 5]{lsY99}})] \label{Young}
There exists $L>0$ such that for all $i, n$ with $0 \leq i \leq n$ and for all $x,y \in [x_{n+1}, x_n]$, we have
\[
| \log (f^i)'(x) - \log (f^i)'y | \leq L \frac{| f^ix - f^iy|}{|x_{n-i}-x_{n-i+1}|} \leq L.
\]
\end{lemma}
In the remarks preceding that lemma, Young shows that there exists $E_1 > 0$ such that
\begin{equation}\label{eqn:deltaxn}
x_n - x_{n+1} \leq E_1 n^{-(1+1/\eps)}.
\end{equation}
Lemma~\ref{Young} and~\eqref{eqn:deltaxn} are the only two results we use from~\cite{lsY99}; in particular, we do not use any of the results that rely on building towers.
 \begin{proof*}[of Lemma \ref{lem:phbbowen}]
We start by estimating the variation on words of the form $0^n z$, where $z\in\{1,\dots,b-1\}$.  Suppose $x, y \in \pi(0^{n} z)$; then by Lemma \ref{Young},
\[
|S_n \ph (x) - S_n \ph(y)| \leq L_1 |f^nx-f^ny|,
\]
where $L_1 = L |x_0-x_1|^{-1}$.  Furthermore, since $f^n x, f^ny \in\pi(z)$, and $f$ is uniformly $C^2$ on each $\pi(z)$ with $z\neq 0$, there exists $L_2$ so that  
\begin{equation} \label{eqn:distort}
|S_{n+1} \ph (x) - S_{n+1} \ph(y)| \leq L_2 |f^{n+1}x-f^{n+1}y|.
\end{equation}
Note also that since $\phg$ is constant on $\pi(0^{n} z)$, we have
\[
|S_n \phb (x) - S_n \phb(y)| = |S_n \ph (x) - S_n \ph(y)|. 
\]
Now we consider more general words.  Take $\lambda >1$ such that $f$ expands distances by a factor of at least $\lambda$ on each interval $I_j$ with $j \neq 0$.  Thus given any word $w\in\LLL(\Sigma_\beta)$ with $k$ non-zero entries, we have
\begin{equation}\label{eqn:diam}
\diam(\pi(w)) \leq \lambda^{-k}.
\end{equation}
Fix $w \in \LLL(\Sigma_\beta)_N$ and write
\[
w=0^{\ell_1} z_1 0^{\ell_2} z_2 \cdots z_m 0^{\ell_{m+1}},
\]
where $\ell_j\geq 0$ and $z_j \in \{1,\dots,b-1\}$ for every $j$. Let $x,y \in \pi(w)$. 
Then, letting $k_1 =0$ and $k_j = (j-1) + \sum_{i=1}^{j-1} \ell_i$, and using \eqref{eqn:distort} and \eqref{eqn:diam}, we have
\begin{align*}
|S&_{N-\ell_{m+1}} \phb(x) - S_{N-\ell_{m+1}} \phb(y)| \leq \sum_{j=1}^m | S_{\ell_j +1} \ph (f^{k_j} x) - S_{\ell_j +1} \ph (f^{k_j} y) |\\ 
&\leq L_1 \sum_{j=1}^m |f^{k_j+1}x - f^{k_j+1} y|
\leq L_1 \sum_{j=1}^m  \diam (\pi(0^{\ell_{j+1}} z_{j+1} \cdots z_m 0^{\ell_{m+1}})) \\
&\leq L_1\sum_{j=1}^m  \lambda^{-(m-j)}  \leq L_1 (1-\lambda^{-1})^{-1} =: L_3.
\end{align*}
Now let $x' = f^{N-\ell_{m+1}}x$ and $y' = f^{N-\ell_{m+1}}y$. Then $x', y' \in \pi(0^{N-\ell_{m+1}})$. Suppose $x' \in \pi(0^{n_1}z_1)$ and $y' \in \pi(0^{n_2} z_2)$ for some $z_1, z_2 >0$. Then, applying Lemma \ref{Young},
\[
|S_{\ell_{m+1}} \phb (x')| = |  S_{\ell_{m+1}} \ph (x') - S_{\ell_{m+1}} \ph(x_{n_1-1})| \leq L,
\]
and similarly for $y'$. 
We conclude that
\[
|S_N \phb (x) - S_N \phb(y)| \leq L_3 + |S_{\ell_{m+1}} \phb (x')| +|S_{\ell_{m+1}} \phb (y')| \leq L_3 +2L,
\]
which shows that $\phb$ has the Bowen property. To see that $\ph$ itself is not Bowen, it suffices to show that 
\[
\sup\{|S_n \ph (x) - S_n \ph (y)| \mid x, y \in \pi(0^n), n \in \NN\} \leq \left| \sum \ph(x_n)\right| = \infty.
\]
This follows quickly from the observation that 
\[
\singlebox
\left\lvert\sum_{k=1}^n \ph(x_k)\right\rvert = |(f^n)'(x_n)| \approx x_n^{-1} \to \infty.
\esinglebox \qedhere
\]
\end{proof*}
\begin{proof}[ of Lemma \ref{lem:posexp}]
To prove Lemma~\ref{lem:posexp}, we use an argument communicated to us by Dmitry Dolgopyat.  Lemma~\ref{lem:WaltersonG} shows that $\ph \in \WWW(\GGG)$, and thus there exists $V>0$ such that $|S_n (\log f')(x) - S_n(\log f')(y)| \leq V$ whenever $x,y\in \pi(w)$ for some $w\in \GGG_n$.  Using properties of logs and exponentiating gives the following bounded distortion result, where we write $E_0 = e^V$:
\begin{equation}\label{eqn:bdddist}
\frac 1{E_0} \leq \frac{(f^n)'(x)}{(f^n)'(y)} \leq E_0.
\end{equation}

Given $x\in\Sigma_\beta$, let $\ell_j=\ell_j(x) \geq 0$ and $z_j \in A\setminus \{0\}$ be such that
\[
x=0^{\ell_1} z_1 0^{\ell_2} z_2 \cdots.
\]
Because $f$ is uniformly expanding away from $0$, it is easy to see that $\pi(x)\in I\setminus Z_0$ whenever $\ulim \frac 1n \ell_n(x) < \infty$.  Thus it suffices to show that there exists $L<\infty$ such that $\hat Z = \{ \pi(x) \mid \ulim \frac 1n \ell_n(x) > L \}$ has zero Lebesgue measure.  

To this end, consider the set
\[
\Omega(m_1,m_2,\dots,m_n) = \{ \pi(x) \mid \ell_j(x) = m_j \text{ for all } 1\leq j\leq n \}.
\]
It follows from~\eqref{eqn:deltaxn} that
\begin{equation}\label{eqn:LebOmega1}
\Leb\Omega(m) \leq E_1 m^{-(1+1/\eps)}
\end{equation}
for every $m$; we claim that there is a constant $E_2$ such that
\begin{equation}\label{eqn:LebOmega}
\Leb(\Omega(m_1,m_2,\dots,m_n)) \leq E_2^n (m_1 \cdots m_n)^{-(1+1/\eps)}.
\end{equation}
Using the bounded distortion estimate~\eqref{eqn:bdddist} and the observation that \[\Leb(f^{m_1 + \cdots + m_n + n}(\Omega(m_1,\dots,m_n)))\] is uniformly bounded below, we conclude that there are constants $E_3, E_4$ such that
\begin{align*}
\frac{\Leb \Omega(m_1,\dots, m_n,m_{n+1})}{\Leb\Omega(m_1,\dots,m_n)}
&\leq E_3\frac{\Leb f^{m_1 + \cdots + m_n + n}\Omega(m_1,\dots, m_n,m_{n+1})}{\Leb f^{m_1 + \cdots + m_n + n}\Omega(m_1,\dots,m_n)} \\
&\leq E_4 \Leb\Omega(m_{n+1}).
\end{align*}
Using~\eqref{eqn:LebOmega1} and iterating gives~\eqref{eqn:LebOmega}.  

Choose $N$  such that $\sum_{m=N}^\infty E_2 m^{-(1+1/\eps)} < 1$ and let $X_i$ be i.i.d.\ $\NN$-valued random variables such that $P(X_i = m) = E_2 m^{-(1+1/\eps)}$ for every $m>N$, and $P(X_i = N) = 1 - \sum_{m>N} E_2 m^{-(1+1/\eps)}$.  Then we have
\begin{align*}
\Leb \Omega(m_1,\dots,m_n) &\leq E_2^n \prod_{j=1}^n m_j^{-(1+1/\eps)} 
\leq \prod_{m_j > N} P(X_j = m_j) \\
&= P(X_j = m_j \text{ for all } j \text{ such that } m_j > N),
\end{align*}
and in particular, for $L>N$ we have
\begin{align*}
\Leb \Big\{ &\pi(x) \,\Big|\, \frac 1n \sum_{j=1}^n \ell_j(x) > L \Big \} \\
&\leq \sum_{m_1 + \cdots + m_n > nL} P(X_j = m_j \text{ for all } j \text{ such that } m_j > N) \\
&\leq P(X_1 + \cdots + X_n > n(L-N)).
\end{align*}
By the law of large numbers, this goes to $0$ as $n\to \infty$ provided we take
\[
L-N > \sum_{m\geq 1} m P(X_j = m) = NP(X_j = N) + \sum_{m > N} E_2 m^{-1/\eps},
\]
which we can do since the sum converges.  This completes the proof of Lemma~\ref{lem:posexp}.
\end{proof}

\begin{proof}[of Lemma \ref{lem:gpposent}]
We begin with the observation that
\[
\Lambda_n(\LLL,t\ph) = \sum_{w\in\LLL_n} \sup_{x\in \pi(w)} (f^n)'(x)^{-t}.
\]
As a consequence of Lemma~\ref{lem:posexp}, there exists a set $Z^+\subset I$, an integer $N\in\NN$, and a number $\lambda>1$ such that $(f^n)'(x) \geq \lambda^n$ for every $x\in Z^+$ and $n\geq N$, and moreover $\Leb Z^+ > 0$.  Let $\DDD = \{w\in \LLL \mid \pi(w) \cap Z^+ \neq \emptyset\}$, and for each $w\in \DDD$ fix a point $x_w\in \pi(w) \cap Z^+$.  Thus for every $n\geq N$, we have
\begin{align*}
\Lambda_n(\LLL,t\ph) &\geq \Lambda_n(\DDD,t\ph) \geq \sum_{w\in\DDD_n} (f^n)'(x_w)^{-t} e^{-V_n(\LLL,S_n(t\ph))} \\
&\geq \sum_{w\in\DDD_n} (f^n)'(x_w)^{-1} \lambda^{n(1-t)} e^{-V_n(\LLL,S_n(t\ph))}.
\end{align*}
Furthermore, we have
\[
\Leb(\pi(w)) \leq \Leb(f^n(\pi(w))) \left(\inf_{x\in\pi(w)} (f^n)'(x)\right)^{-1} 
\leq (f^n)'(x_w)^{-1} e^{V_n(\LLL,S_n\ph)},
\]
and so
\begin{align*}
\Lambda_n(\LLL,t\ph) &\geq \sum_{w\in\DDD_n} \Leb(\pi(w)) \lambda^{n(1-t)} e^{-(1+|t|)V_n(\LLL,S_n\ph)} \\
&\geq \Leb(Z^+) \lambda^{n(1-t)} e^{-(1+|t|)V_n(\LLL,S_n\ph)}.
\end{align*}
This yields
\[
\frac 1n \log \Lambda_n(\LLL,t\ph) \geq (1-t)\log\lambda + \frac 1n \log \Leb(Z^+)
-\frac{1+|t|}{n} V_n(\LLL,S_n\ph),
\]
and since $\ph$ is continuous the final term goes to $0$ as $n\to\infty$, whence
\begin{equation}\label{eqn:negderivative}
P(\LLL,t\ph) \geq (1-t)\log\lambda > 0
\end{equation}
for every $t<1$.  
\end{proof}

\bibliographystyle{amsalpha}
\bibliography{master}

\end{document}